\documentclass[11pt,twoside]{article}

\usepackage{amssymb}
\usepackage{amsmath}
\usepackage{amsthm}
\usepackage{color}

\allowdisplaybreaks

\def\rr{{\mathbb R}}
\def\rn{{{\rr}^n}}

\def\zz{{\mathbb Z}}
\def\nn{{\mathbb N}}

\def\cd{{\mathcal D}}

\def\cf{{\mathcal F}}

\def\cb{{\mathcal B}}

\def\cx{{\mathcal X}}

\def\fz{\infty}
\def\az{\alpha}
\def\supp{{\mathop\mathrm{\,supp\,}}}

\def\lz{\lambda}
\def\dz{\delta}

\def\ez{\epsilon}

\def\bz{\beta}

\def\wz{\widetilde}

\def\ls{\lesssim}
\def\gs{\gtrsim}

\def\tbz{{\triangle_\lz}}
\def\dmz{{dm_\lz}}

\def\riz{{R_{\Delta_\lz}}}

\def\rrp{{{\mathbb\rr}_+}}
\def\bmoz{{{\rm BMO}(\mathbb\rrp,\, dm_\lz)}}
\def\cmoz{{{\rm CMO}(\mathbb\rrp,\, dm_\lz)}}

\def\inzf{{\int_0^\fz}}

\def\xtz{{x^{2\lz}\,dx}}
\def\ytz{{y^{2\lz}\,dy}}
\def\ztz{{z^{2\lz}\,dz}}

\def\mofi{{M_\lz(f, I)}}

\def\loz{{L^1(\rrp,\, dm_\lz)}}
\def\lpz{{L^p(\rrp,\, dm_\lz)}}

\def\linz{{L^\fz(\rrp,\, dm_\lz)}}

\def\dint{\displaystyle\int}

\def\dfrac{\displaystyle\frac}

\def\r{\right}
\def\lf{\left}

\def\beeqn{\begin{equation}}
\def\eneqn{\end{equation}}
\def\beeqns{\begin{equation*}}
\def\eneqns{\end{equation*}}
\def\beeqa{\begin{eqnarray}}
\def\eneqa{\begin{eqnarray}}
\def\beeqas{\begin{eqnarray*}}
\def\eneqas{\begin{eqnarray*}}
\def\besp{\begin{split}}
\def\ensp{\begin{split}}
\def\noz{\nonumber}

\newtheorem{thm}{Theorem}[section]
\newtheorem{lem}[thm]{Lemma}%[section]
\newtheorem{prop}[thm]{Proposition}%[section]
\newtheorem{rem}[thm]{Remark}%[section]
\newtheorem{defn}[thm]{Definition}%[section]
\numberwithin{equation}{section}

\begin{document}

\arraycolsep=1pt

\title{\Large\bf  Compactness of Riesz transform commutator associated with Bessel operators}
\author{Xuan Thinh Duong, Ji Li, Suzhen Mao, Huoxiong Wu and Dongyong Yang\,\footnote{Corresponding author}}

%\medskip
\date{}
\maketitle

\begin{center}
\begin{minipage}{13.5cm}\small

{\noindent  {\bf Abstract:}\  Let $\lambda>0$  and
$\triangle_\lambda:=-\frac{d^2}{dx^2}-\frac{2\lambda}{x}
\frac d{dx}$ be the Bessel operator on $\mathbb R_+:=(0,\infty)$. We first introduce and
obtain an equivalent characterization of ${\rm CMO}(\mathbb R_+,\, x^{2\lambda}dx)$.
By this equivalent characterization and establishing a new version of the Fr\'{e}chet-Kolmogorov theorem
in the Bessel setting, we further prove that a function
$b\in {\rm BMO}(\mathbb R_+,\, x^{2\lambda}dx)$ is in ${\rm CMO}(\mathbb R_+,\, x^{2\lambda}dx)$
if and only if the Riesz transform commutator $[b, R_{\Delta_\lambda}]$ is
compact on $L^p(\mathbb R_+, x^{2\lambda}dx)$ for any $p\in(1, \infty)$.}

\end{minipage}
\end{center}

\bigskip
\bigskip


\medskip


\medskip

\thanks{X. T. Duong is supported by ARC DP 140100649.}

\thanks{J. Li is supported by ARC DP 160100153 and Macquarie University New Staff Grant.}

\thanks{H. Wu is supported by the NNSF of China (Grant Nos. 11371295, 11471041) and the NSF of Fujian Province of China (No. 2015J01025).}

\thanks{D. Yang is supported by the NNSF of China (Grant No. 11571289) and the State Scholarship Fund of China (No. 201406315078).}

\section{Introduction and Statement of Main Results\label{s1}}

Let $\lz$ be a positive constant
%\in\rrp:=(0, \fz)$
and $\tbz$ be the Bessel operator defined by
by setting, for  suitable functions $f$ and  $x\in \rrp :=(0, \fz)$,
\begin{equation*}%\label{bessel 1}
\tbz f(x):=-\frac{d^2}{dx^2}f(x)-\frac{2\lz}{x}\frac{d}{dx}f(x);
\end{equation*}
see \cite{bdt,ms}.
An early work concerning the Bessel operator is from Muckenhoupt and Stein \cite{ms}.
They aimed to  develop a theory associated to
$\tbz$ which is parallel to the classical one associated to the Laplace operator $\triangle$.
After that, a lot of work concerning the Bessel operators was carried out. See, for example \cite{ak,bcfr,bfbmt,bfs,bhnv, dlwy, k78,v08,yy} and the references therein. Among the study of $\tbz$,
the properties of Riesz transforms  associated  to $\tbz$   defined by
\begin{equation*}%\label{riesz}
R_{\Delta_\lambda}f := \partial_x (\Delta_\lz)^{-1/2}f,
\end{equation*}
have been studied extensively, see for example \cite{ak,bcfr,bfbmt, ms,v08}.
Characterizations of function spaces associated to the Bessel operator $\tbz$ were also studied
by many authors. Among these, we point out that the Lebesgue space
associated to the Bessel operator $\tbz$ is of the form $\lpz$,
where $1<p<\infty$, $\dmz(x):= x^{2\lz}\,dx$, and $dx$ is the standard Lebesgue measure on $\mathbb{R}$ (see for example \cite{bdt}).
Moreover, in \cite{bdt}, Betancor et al.
characterized the Hardy space $H^1(\mathbb{R}_+, \dmz)$
associated to $\tbz$ in terms of the Riesz transform and the radial
maximal function associated with the Hankel convolution of a class
of suitable functions. More recently, Duong et al. \cite{dlwy} established a factorisation of the Hardy space associated to $\tbz$
and a characterisation of the BMO space associated to $\tbz$ through commutators $[b,\riz]$, which is defined as follows:
\begin{equation*}%\label{commutator}
[b, \riz]f(x):= b(x)\riz f(x)-\riz(bf)(x),
\end{equation*}
where $b\in L^1_{\rm loc}(\mathbb{R}_+,\,dm_\lz)$ and $f\in \lpz$.

The aim of this paper is to provide a characterization of the compactness of the Riesz commutator $[b, \riz]$, based on
the characterization in \cite{dlwy}.

We recall that
the first result on  characterization of compactness of commutators of singular integrals is due to Uchiyama \cite{u78}.
He refined the $L^p$-boundedness results of Coifman et al. \cite{crw} on the commutator with the symbol $b$ in the space BMO to compactness.
This is achieved by requiring the symbol $b$ to be not just in BMO, but rather in CMO,
 which is the closure in
BMO of the space of $C^\infty$ functions with compact supports.
Since then, many authors focused on the compactness of commutators with certain singular integrals,
including linear, nonlinear and bilinear operators on variant function spaces. See for example \cite{bl, bt, cc, ch, h, i,  kl1, kl2, msw}
and the references therein.

We further note that the compactness of the commutator has extensive applications in partial differential equations,
see for example the application to
$\bar{\partial}-$Neumann problem on forms \cite[Chapter 12, Section 8]{t}.
Moreover, to study the $L^p$-theory of quasiregular mappings, Iwaniec \cite{i}
considered the linear complex Beltrami equation and derived the $L^p(\mathbb C)$-invertibility of Beltrami operator $I-\mu T$, via
the compactness of the commutator $[\mu, T]$ on $L^p(\mathbb C)$ and the index theory of Fredholm operators on Banach spaces,
where the Beltrami coefficient $\mu\in L^\fz(\mathbb C)\cap {\rm CMO}(\mathbb C)$ has  compact support
 and $T$ is the Beurling-Ahlfors singular integral operator. In their remarkable work \cite{ais},
Astala et al. further extended the result in \cite{i} by removing the restrictive assumption on the compact support of
$\mu$.  Recently, based on the result of Iwaniec \cite{i},
Clop and Cruz \cite{cc} obtained a priori  estimate in $L^p(\omega)$ for the generalized Beltrami equation
and regularity for the Jacobian of certain quasiconformal mappings, where
the weight $\omega$ belongs to the  Muckenhoupt class $A_p$.
See also \cite{mov,cmo} for the application of the compactness of commutator generated by Beurling-Ahlfors transform  and
${\rm CMO}$ functions to  Beltrami equations.

Before stating our main result, we first recall the definition of the BMO space associated with the Bessel operator which is known to coincide with
the standard BMO space on $(\rrp,dm_\lambda)$. For every $x$, $r\in \mathbb{R}_+$, we define $I(x, r):=(x-r, x+r)\cap \mathbb{R}_+$.
\begin{defn}[\cite{yy}]\label{d-bmo}
A function $f\in L^1_{\rm loc}(\rrp,dm_\lambda)$ belongs to
the {\it space} $\bmoz$ if
\begin{equation*}\label{mofi}
%\sup_{x,\,r\in(0,\,\fz)}\mofi:=\frac1{m_\lz(I)}\int_I\lf|f(x)-\mfi\r|\xtz<\fz,
\sup_{x,\,r\in (0,\,\infty)}M_\lambda(f,I(x,r)):=\sup_{x,\,r\in (0,\,\infty)}
\frac 1{m_\lambda(I(x,r))}\int_{I(x,\,r)}|f(y)-f_{I(x,\,r),\,\lambda}|y^{2\lambda}dy<\infty,
\end{equation*}
where
\begin{equation}\label{average}
%\mfi:=\frac1{m_\lz(I)}\int_I f(x)\,\xtz.
f_{I(x,\,r),\,\lambda}:=\frac 1{m_{\lambda}(I(x,\lambda)}\int_{I(x,\,r)}f(y)y^{2\lambda}dy.
\end{equation}
\end{defn}
We further denote by $\cmoz$
the $\bmoz$-closure of $\cd$, the set of $C^\fz(\rrp)$ functions with compact supports.

The main result of this paper is stated as follows:
\begin{thm}\label{t-riesz compact}
Let $b\in\bmoz$. Then $b\in \cmoz$ if and only if the Riesz transform commutator $[b, R_{\Delta_\lambda}]$ is
compact on $L^p(\mathbb{R}_+,\, dm_\lz)$ for any $p\in(1, \infty)$.
\end{thm}

%\begin{rem}\rm
%\end{rem}
The  proof of Theorem \ref{t-riesz compact} is carried out from Section \ref{s2} to \ref{s5}, and contains the following ingredients:

(i) The doubling and reverse doubling properties of the space $(\mathbb R_+,\, x^{2\lambda}dx)$.
More specifically, in Section \ref{s2},
we first prove that the measure $dm_\lz$ satisfies a doubling property with constant $2^{2\lz+1}$
and reverse doubling property with constant $\min(2, 2^{2\lz})$ (see Proposition \ref{p-redoubl} below).
We remark that the  constants are almost sharp in the sense that
Proposition \ref{p-redoubl} is false if $\min(2, 2^{2\lz})$ is replaced by 2 or
$2^{2\lz+1}$ replaced by $\max(2, 2^{2\lz})$; see Remark \ref{r-reverse doubl} below.

(ii) Kernel bound estimates  of the Riesz transforms (Lemma \ref {l-RieszCZ}  and Proposition \ref {p-lower bdd riesz}).
We recall some known upper and lower bounds as well as the H\"older's regularity of the kernel $\riz(x, y)$ of Riesz transform $\riz$
and establish a new estimate of the lower bound of $\riz(x, y)$, which plays a key role in the proof of the main result.

(iii) A new characterisation of the space ${\rm CMO}(\mathbb R_+,\, x^{2\lambda}dx)$ which is also of independent interest (Theorem \ref{t-cmo char} in
Section 3). We employ the idea of Uchiyama \cite{u78}. However, since the space $\lpz$ is not invariant under translations,
we need some new techniques and adapt the proof in \cite{u78} to our setting.

(iv) A new version of the Fr\'echet-Kolmogorov theorem in the Bessel setting (Theorem \ref{t-fre kol} in Section \ref{s4}). We
remark that Clop and Cruz \cite{cc} obtained a partial result of  Fr\'{e}chet-Kolmogorov theorem when $\omega$ belongs to
$A_p(\rn)$ with the Lebesgue measure and $p\in(1, \fz)$. However, in current setting,
the weight $x^{2\lz}$ for general $\lz\in(0, \fz)$ might not belong to $A_p(\rrp)$.

(v) Estimates on the commutator $[b, R_{\Delta_\lambda}]$ are carried out in Section \ref{s5}. By the upper and lower bounds and
the H\"older's regularity  of $\riz(x, y)$ in Section \ref{s2},
we first obtain a lemma for the upper and lower bounds of integrals of $[b, \riz]f_j$ on certain intervals, for $b\in\bmoz$ and proper function $f_j$.
%f_j\}_j$ is a bounded subset of $\lpz$
%and $b\in\bmoz$.
Using this and
a contradiction argument in terms of the aforementioned equivalent characterization of $\cmoz$ in Section \ref{s3},
we show that if $[b,\riz]$ is compact on $\lpz$, then $b\in\cmoz$.

(vi)
By the upper bound and the H\"older's regularity  of $\riz$, together
with the Fr\'{e}chet-Kolmogorov theorem in Section \ref{s4}, we show via a density argument that if $b\in\cmoz$, then $[b,\riz]$ is compact on $\lpz$.

Here, for the necessity, we also note that Krantz and Li \cite{kl2} showed that if $T$ is a singular integral operator bounded on $L^2(\mathcal X)$ and $b\in{\rm CMO}(\mathcal X)$, then $[b,\,T]$ is compact on $L^p(\mathcal X)$ for all $p\in(1,\,\infty)$.
However, the underlying space $(\mathcal X,\,d,\,\mu)$ studied in \cite{kl2} is a space of homogeneous type which satisfies the following condition: there exist positive constants $C$ and $\epsilon_0\in(0,\,1)$ such that for all $x,\,y\in \mathcal X$ and $d(x,\,y)\leq r\leq1$,
\begin{equation}\label{mu-meas-condi}
\mu(B(x,\,r)\setminus B(y,\,r))+\mu(B(y,\,r)\setminus B(x,\,r))\leq C\lf(\frac{d(x,\,y)}{r}\r)^{\epsilon_0}.
\end{equation}

We point out that the underlying space $(\mathbb R_+,\,|\cdot|,\,dm_\lz)$ in the Bessel setting does not fall into the scope of
the space $(\mathcal X,\,d,\,\mu)$ studied by \cite{kl2}.
In fact, let $x:=N+1$, $y:=N$, $r:=1$. Then we see that
$$m_\lz(I(N+1,\,1)\setminus I(N,\,1))=\int_{N+1}^{N+2}x^{2\lz}dx\geq(N+1)^{2\lz}\rightarrow\infty,$$
as $N\rightarrow\infty$, and that
$$  \frac{d(x,\,y)}{r}=1.  $$
Hence \eqref{mu-meas-condi} is not true for our $(\mathbb R_+,\,|\cdot|,\,dm_\lz)$.

\smallskip

Throughout the paper,
we denote by $C$ and $\widetilde{C}$ {positive constants} which
are independent of the main parameters, but they may vary from line to
line. For every $p\in(1, \fz)$, $p'$ means the conjugate of $p$, i.e., $1/p'+1/p=1$.
If $f\le Cg$, we then write $f\ls g$ or $g\gs f$;
and if $f \ls g\ls f$, we  write $f\sim g.$
For any $k\in \mathbb{R}_+$ and interval $I:= I(x, r)$ for some $x$, $r\in (0, \fz)$,
$kI:=I(x, kr)$ and $I+y:=\{x+y:\,\,x\in I\}$.
For any $x,\,r\in(0,\,\infty)$, if  $x<r$, then
$$I(x,\,r)=(0,\,x+r)=I\lf(\frac{x+r}{2},\,\frac{x+r}{2}\r).$$
Thus, for a given interval $I(x, r)$, without any specific condition, we may always assume that
\begin{equation}\label{assum interval}
x\ge r.
\end{equation}

Moreover, for any $i\in\zz$, let
\begin{equation}\label{i-gener}
R_i:=\{x\in\rrp:\,\, x\le 2^i\}.
\end{equation}
If $I$ is a dyadic interval in $\rrp$, then $I=(k2^j, (k+1)2^j]$, where $k\in\zz_+,$ and $j\in\zz$.

\section{Preliminaries: Reverse doubling property  and bounds of Riesz tranforms}\label{s2}

In this section, we present some preliminary results, including the reverse doubling property of $dm_\lz$ and
upper and lower bound of Riesz transform $\riz$.
%
%We first point out that from the definition of  $m_\lambda$ (i.e., $dm_\lambda(x):=x^{2\lambda}dx$)  that
% there exists a finite constant $C > 1$ such that for all $x,\,r\in\mathbb{R}_+$,
%\begin{equation}\label{volume property-1}
% C^{-1}m_\lz(I(x, r))\le x^{2\lz}r+r^{2\lz+1}\le C m_\lz(I(x, r)).
%\end{equation}
%Hence the space $(\rrp,dm_\lambda)$ is a space of homogeneous type in the sense of Coifman and Weiss \cite{cw77}.
%Therefore, the standard Hardy space and BMO space can be defined for $(\rrp,dm_\lambda)$
%(in the sense of Coifman and Weiss  \cite{cw77}).
%
%
%
We begin with the following proposition,
which implies that $(\mathbb{R}_+, \rho, dm_\lz)$ is an RD space in \cite{hmy},
where $\rho(x,y):=|x-y|$ for all $x,\,y\in\rrp$.
\begin{prop}\label{p-redoubl}
For any $I\subset \rrp$,
\begin{equation}\label{reverse doubl}
\min\big(2, 2^{2\lz}\big)m_\lz(I)\le m_\lz(2I)\le 2^{2\lz+1}m_\lz(I).
\end{equation}
\end{prop}

\begin{proof}
Let $I:=I(x,r)\subset\rrp$ be an interval. As in the argument of \eqref{assum interval}, we can assume that $x\ge r$.
Moreover, since when $x=r$, it is easy to see that $m_\lz(2I)=(3/2)^{2\lz+1}m_\lz(I)$ and hence Proposition \ref{p-redoubl} holds.
Thus, we further assume that $x>r$.
Observe that
\begin{equation*}
m_\lz(I):=\int_{x-r}^{x+r}y^{2\lz}dy=\frac{(x+r)^{2\lz+1}-(x-r)^{2\lz+1}}{2\lz+1}
\end{equation*}
and
\begin{equation*}
m_\lz(2I)=\left\{
  \begin{array}{ll}
    \dfrac{(x+2r)^{2\lz+1}-(x-2r)^{2\lz+1}}{2\lz+1}, & \hbox{$x\ge 2r$;} \\
    \dfrac{(x+2r)^{2\lz+1}}{2\lz+1}, & \hbox{$r< x<2r$.}
  \end{array}
\right.
\end{equation*}
Let $t:=r/x$,
\begin{equation*}
f_\lz(t):=\left\{
            \begin{array}{ll}
             (1+2t)^{2\lz+1}-(1-2t)^{2\lz+1}-\min\lf(2, 2^{2\lz}\r)\lf[(1+t)^{2\lz+1}-(1-t)^{2\lz+1}\r], &\, \hbox{$t\in[0, 1/2]$,} \\
             (1+2t)^{2\lz+1}-\min\lf(2, 2^{2\lz}\r)\lf[(1+t)^{2\lz+1}-(1-t)^{2\lz+1}\r], & \hbox{$t\in(1/2,1)$;}
            \end{array}
          \right.
\end{equation*}
and
\begin{equation*}
 \tilde f_\lz(t):=\left\{
  \begin{array}{ll}
(1+2t)^{2\lz+1}-(1-2t)^{2\lz+1}-2^{2\lz+1}[(1+t)^{2\lz+1}-(1-t)^{2\lz+1}], &\, \hbox{$t\in[0, 1/2]$,} \\
    (1+2t)^{2\lz+1}-2^{2\lz+1}\lf[(1+t)^{2\lz+1}-(1-t)^{2\lz+1}\r], & \hbox{$t\in(1/2,1)$.}
  \end{array}
\right.
\end{equation*}
To show \eqref{reverse doubl}, it suffices to prove that for all $t\in(0, 1)$ and $\lz\in(0, \fz)$,
\begin{equation*}
f_\lz(t)\ge 0 \,\,{\rm and}\,\, \tilde f_\lz(t)\le0.
\end{equation*}

We first prove that $f_\lz(t)\ge 0$ by considering the following four cases:

Case (i)  $t\in(0,\,1/2]$ and $\lz\in(0,\,1/2]$. In this case,
$$f_\lz'(t)=(2\lz+1)\lf\{2(1+2t)^{2\lz}+2(1-2t)^{2\lz}-2^{2\lz}\lf[(1+t)^{2\lz}+(1-t)^{2\lz}\r]\r\}.$$
Observe that the function $g(t):=t^{2\lz}$ is a concave function of $t$ for given $\lz\in(0,\,1/2]$. By the fact that for any $a,\,b\in(0, \fz)$,
$$(a+b)^{2\lz}\leq a^{2\lz}+b^{2\lz}\leq2^{1-2\lz}(a+b)^{2\lz},$$
we see that
$$f_\lz'(t)\ge(2\lz+1)\lf[2(1+2t+1-2t)^{2\lz}-2^{2\lz}2^{1-2\lz}(1+t+1-t)^{2\lz}\r]=0,$$
 which further implies that $f_\lz(t)\ge f_\lz(0)=0$ for any $t\in(0,\,1/2]$.

Case (ii) $t\in(0,\,1/2]$ and $\lz\in(1/2,\,\infty)$. In this case,
$$f_\lz'(t)=2(2\lz+1)\lf\{(1+2t)^{2\lz}+(1-2t)^{2\lz}-\lf[(1+t)^{2\lz}+(1-t)^{2\lz}\r]\r\}.$$
By the fact that $g(t):=t^{2\lz}$ is a convex function for $\lz\in(1/2,\,\infty)$, we get that for any $t\in(0,\,1/2],$
$$(1+t)^{2\lz}=\lf[\frac{1-2t}{4}+\frac{3(1+2t)}{4}\r]^{2\lz}\leq\frac{(1-2t)^{2\lz}}{4}+\frac{3(1+2t)^{2\lz}}{4},$$
and
$$(1-t)^{2\lz}=\lf[\frac{3(1-2t)}{4}+\frac{1+2t}{4}\r]^{2\lz}\leq\frac{3(1-2t)^{2\lz}}{4}+\frac{(1+2t)^{2\lz}}{4}.$$
Combining these inequalities above, we see that $f_\lz'(t)\ge0$  and hence $f_\lz(t)\geq f_\lz(0)=0$ for any $t\in(0,\,1/2]$.

Case (iii) $t\in(1/2, 1)$ and $\lz\in(0,\,1/2]$.
In this case,
$$f_\lz(t)=(1+2t)^{2\lz+1}-2^{2\lz}\lf[(1+t)^{2\lz+1}-(1-t)^{2\lz+1}\r].$$
To show $f_\lz(t)\ge0$ for all $t\in(1/2, 1)$ and $\lz\in(0,\,1/2]$, it suffices to prove that
$$g_\alpha(t):=2^{-\alpha}-\lf[\lf(\frac{1+t}{1+2t}\r)^{\alpha+1}-\lf(\frac{1-t}{1+2t}\r)^{\alpha+1}\r]\ge0$$
for all $t\in(1/2, 1)$ and $\alpha:=2\lz\in(0,\,1]$. On the other hand, observe that for fixed $t\in(\frac12, 1)$,
$g_\alpha(t)$ is decreasing in $\alpha\in(0,\,1]$,
which implies that for any $t\in(1/2,\,1)$,
$$g_\alpha(t)\geq g_1(t)=\frac{(1-2t)^2}{2(1+2t)^2}\geq0.$$
Thus, we conclude that $f_\lz(t)\geq0$ for all $t\in(1/2,\,1)$.

Case (iv) $t\in(1/2, 1)$ and $\lz\in(1/2,\,\infty)$. In this case,
we also have $f'_\lz(t)\geq0$ for all $t\in(1/2,\,1)$.
Therefore $f_\lz(t)$ is increasing in $t\in(1/2, \,1)$, and
$$f_\lz(t)\geq f_\lz(1/2)=2^{2\lz+1}-2\lf[(3/2)^{2\lz+1}-(1/2)^{2\lz+1}\r]\geq0$$
for all $\lz\in(1/2,\,\infty)$.

Combining the four cases above, we conclude that $f_\lz(t)\geq0$ for all $\lz\in(0,\,\infty)$
and $t\in(0, \,1)$.

Now we show $\tilde f_\lz(t)\le0$ for all $t\in(0,\,1)$ and $\lz\in(0,\,\infty)$. Similarly,
when $t\in(0,\,1/2]$, we have
$$\tilde f_\lz'(t)=(2\lz+1)\lf\{2(1+2t)^{2\lz}+2(1-2t)^{2\lz}-2^{2\lz+1}\lf[(1+t)^{2\lz}+(1-t)^{2\lz}\r]\r\}<0,$$
 and $\tilde f_\lz(t)\le \tilde f_\lz(0)=0$ for all $t\in(0,\,1/2]$.
When  $t\in[1/2,\,1)$,
$$\tilde f_\lz'(t)=(2\lz+1)\lf\{2(1+2t)^{2\lz}-2^{2\lz+1}\lf[(1+t)^{2\lz}+(1-t)^{2\lz}\r]\r\}<0,$$
and $\tilde f_\lz(t)\le \tilde f_\lz(1/2)=0$ for all $t\in[1/2,\,1)$.
This finishes the proof of Proposition \ref{p-redoubl}.
\end{proof}

\begin{rem}\label{r-reverse doubl}\rm
We remark that the  constants in \eqref{reverse doubl} are almost sharp in the sense that
\eqref{reverse doubl} is false if $2^{2\lz+1}$ is replaced by $\max(2, 2^{2\lz})$  or
$\min(2, 2^{2\lz})$ replaced by 2. In fact, to see \eqref{reverse doubl} is false if
$2^{2\lz+1}$ is replaced by $\max(2, 2^{2\lz})$, it suffices to take $\lz:=1/2$ and $r:=x$ for any $x\in\rrp$.
On the other hand, to see \eqref{reverse doubl} is false if $\min(2, 2^{2\lz})$ is replaced by 2,
consider the case $\lz\in(0, 1/2]$ and $r:=x/2$ for any $x\in\rrp$. Let
$$h_{1/2}(\lz):=2^{2\lz+1}-2\lf[\lf(\frac32\r)^{2\lz+1}-\lf(\frac12\r)^{2\lz+1}\r],\,\lz\in[0, 1/2].$$
Then $h_{1/2}(0)=0=h_{1/2}(1/2)$ and $h_{1/2}(\lz)=2^{2\lz+1}\tilde h_{1/2}(t)$, where
$$\tilde h_{1/2}(\lz):=1-2\lf[\lf(\frac34\r)^{2\lz+1}-\lf(\frac14\r)^{2\lz+1}\r].$$
Observe that $\tilde h_{1/2}(\lz)\le0$ on $[0, 1/2]$. This implies that $h_{1/2}(\lz)\le 0$ for all $\lz\in(0, 1/2]$
and so
$$m_\lz(I(x, x))-2m_\lz(I(x,x/2))=\frac{r^{2\lz+1}}{2\lz+1}h_{\frac12}(\lz)\le0.$$
\end{rem}

We now recall some known upper and lower bounds of the kernel $\riz(y,z)$ of $\riz$.
The following estimates can be found in, for example, \cite{bfbmt,dlwy}.
\begin{lem}\label{l-RieszCZ}
The kernel $\riz(y,z)$ satisfies the following  estimates:
\begin{itemize}
  \item [{\rm i)}] There exists a positive constant $C$ such that for any $y,z\in\mathbb{R}_+$ with $y\not=z$,
  \begin{equation}\label{cz kernel condition-1}
  |\riz(y,z)|\le C \frac1{m_\lz(I(y, |y-z|))} .
  \end{equation}
  \item [{\rm ii)}]There exists a positive constant $\wz C$ such that  for any $y,\,y_0,\, z\in \mathbb{R}_+$ with $|y_0-z|<|y_0-y|/2$,
  \begin{eqnarray}\label{cz kernel condition-2}
   &&|\riz(y, y_0)-\riz(y,z)|+ |\riz(y_0, y)-\riz(z,y)|\nonumber\\
   &&\quad\le \wz C \frac{|y_0-z|}{|y_0-y|}\frac1{m_\lz(I(y, |y_0-y|))}.
  \end{eqnarray}
%
%  \item [{\rm iii)}] There exist positive constants $K_1>2$ large enough and $C_{K_1,\,\lz}, \wz C_{K_1,\,\lz}\in(1, \fz)$ such that
%  for any $y,\,z\in\mathbb{R}_+$ with $z>K_1y$,
%  \begin{equation*}
%C_{K_1,\,\lz}^{-1}\frac y{z^{2\lz+2}}\le  \riz(y, z)\le \wz C_{K_1,\,\lz}\frac y{z^{2\lz+2}}.
%  \end{equation*}
  \item [{\rm iii)}]There exist $K_1\in(0, 1)$ small enough and a positive constant $C_{K_1,\,\lz}$ such that
  for any $y,\,z\in\mathbb{R}_+$ with $z<K_1y$,
  \begin{equation*}
  \riz(y, z)\le -C_{K_1,\,\lz}\frac1{y^{2\lz+1}}.
  \end{equation*}
  \item [{\rm iv)}] There exist $K_2\in(1/2,1)$ such that $1-K_2$ small enough and a positive constant $C_{K_2,\lz}$
  such that for any $y,\,z\in\mathbb{R}_+$ with $z/y\in(K_2, 1)$,
  \begin{equation*}\label{appro estimate of riesz kernel}
 \lf |\riz(y,z)+\frac1\pi\frac1{y^\lz z^\lz}\frac1{y-z}\r|\le C_{K_2,\,\lz}\frac1{y^{2\lz+1}}\lf(\log_+\frac{\sqrt{yz}}{|y-z|}+1\r).
  \end{equation*}
 % \item[vii)] There exist a positive constant $\wz C_{K_1}\in(1, \fz)$ such that
%  for any $y,\,z\in\mathbb{R}_+$ with $z>2y$,
%  \begin{equation}\label{upper bound of riesz kernel-abs val}
%|\riz(y, z)|\le \wz C_{K_1,\,\lz}\frac y{z^{2\lz+2}}.
%  \end{equation}
\end{itemize}
\end{lem}
\begin{rem}\label{l-RieszCZ-1}\rm
We mention that by Lemma \ref{l-RieszCZ} iv), there exists $\tilde K_2\in(K_2,\,1)$ such that for any $y,\,z\in\mathbb R_+$ with $\tilde K_2<z/y<1$,
\begin{equation}\label{up-bdd-riesz-kernel}
-\riz(y, z)\geq \frac1{2\pi y^{\lz}z^{\lz}(y-z)}.
\end{equation}
In fact, from Lemma \ref{l-RieszCZ} iv), it follows that for any $y,\,z\in\mathbb{R}_+$ with $K_2 <z/y<1$,
\begin{equation*}
-\riz(y, z)\geq \frac{1}{\pi}\frac1{y^{\lz}z^{\lz}}\frac{1}{y-z}-C_{K_2,\,\lz}\frac{1}{y^{2\lz+1}}\lf(\log_+\frac{\sqrt{yz}}{|y-z|}+1\r).
\end{equation*}
To show \eqref{up-bdd-riesz-kernel}, it suffices to prove that there exists $\tilde K_2\in (K_2,\,1)$ such that for all $y,\,z\in\mathbb R_+$ with $z/y\in(\tilde K_2,\,1)$,
$$C_{K_2,\,\lz}\frac{1}{y^{2\lz+1}}\lf(\log_+\frac{\sqrt{yz}}{|y-z|}+1\r)\leq\frac{1}{2\pi}\frac{1}{y^\lz z^\lz}\frac{1}{y-z}.$$
Equivalently, we only need to show that
\begin{equation*}
\lf(\frac{z}{y}\r)^{\lz}\frac{y-z}{y}\lf(\log_+\frac{\sqrt{yz}}{|y-z|}+1\r)\leq\frac{1}{2\pi C_{K_2,\,\lz}}.
\end{equation*}
Note that
$$\lf(\frac{z}{y}\r)^{\lz}\frac{y-z}{y}\lf(\log_+\frac{\sqrt{yz}}{|y-z|}+1\r)
=\lf(\frac{z}{y}\r)^{\lz}\lf(1-\frac{z}{y}\r)\lf(\log_+\frac{\sqrt{z/y}}{|1-z/y|}+1\r)\rightarrow0,$$
as $z/y\rightarrow1^-.$ This implies the existence of $\tilde K_2$, which shows \eqref{up-bdd-riesz-kernel}.
\end{rem}
As a consequence of Lemma \ref{l-RieszCZ} iii) and Remark \ref{l-RieszCZ-1} above, we further establish a new version of lower bound for $\riz(y,z)$
for all $z<y$, which plays a key role in the proof of Theorem \ref{t-riesz compact}.

\begin{prop}\label{p-lower bdd riesz}
 There exists a positive constant $C_0$ such that for any $y,\,z\in\rrp$ with $z<y$,
  $$\riz(y, z)\le -C_0\frac1{m_\lz(I(y, y-z))}.$$
 \end{prop}

\begin{proof} Since $y>z$ and $y>y-z$, we first see that $m_\lz(I(y,\,y-z))\thicksim y^{2\lz}(y-z)$, thus we only need to show
$$-\riz(y, z)\gs \frac1{y^{2\lz}(y-z)}.$$
Recall that
\begin{equation}\label{riesz kernel}
\riz(y,z)=-\dfrac{2\lz}{\pi}\dint_0^\pi\dfrac{(y-z\cos\theta)(\sin\theta)^{2\lz-1}}
{(y^2+z^2-2yz\cos\theta)^{\lz+1}}\,d\theta;
\end{equation}
see, for example, \cite{bdt}.
For any fixed $y, z\in \mathbb{R}_+$ with $y>z$, write $z=sy$. Then
$s\in (0, 1)$. If $s<K_1$, where $K_1$ is as in iii) of Lemma \ref{l-RieszCZ},
 then by iii) of Lemma \ref{l-RieszCZ}, we see that
$$-\riz(y, z)\gs \frac1{y^{2\lz+1}}\gs \frac1{y^{2\lz}(y-z)}.$$
On the other hand, for $\tilde K_2$ as in Remark \ref{l-RieszCZ-1}, and any $y,\,z\in\mathbb R_+$ with $\tilde K_2<y/z<1$,
$$-\riz(y, z)\gs \frac1{y^{\lz}z^{\lz}(y-z)}.$$
If $z=sy$ and $s\in(\tilde K_2,1)$, then
$$-\riz(y, z)\gs \frac1{y^{2\lz}(y-z)}.$$
Thus, it remains to consider $s\in[K_1, \tilde K_2]$. By \eqref{riesz kernel}, we write
\begin{equation*}
\riz(y,sy)=\frac{-2\lz}{\pi}\frac1{y^{2\lz+1}}\dint_0^\pi\dfrac{(1-s\cos\theta)(\sin\theta)^{2\lz-1}}
{(1+s^2-2s\cos\theta)^{\lz+1}}\,d\theta=:\frac{-2\lz}{\pi}\frac1{y^{2\lz+1}}{\rm I}.
\end{equation*}
Since $s\in[K_1, \tilde K_2]$ and $\frac2\pi\theta\le\sin\theta\le\theta$ for any $\theta\in(0, \pi/2)$, we see that
$1-\tilde K_2\le 1-s<1$ and
\begin{eqnarray*}
{\rm I}&=&\dint_0^\pi\dfrac{[(1-s)+s(1-\cos\theta)]}{(1+s^2-2s\cos\theta)^{\lz+1}}(\sin\theta)^{2\lz-1}\,d\theta\\
&\ge&\dint_0^{\frac\pi2}\dfrac{1-s}{[(1-s)^2+2s(1-\cos\theta)]^{\lz+1}}(\sin\theta)^{2\lz-1}\,d\theta\\
&\gs&(1-s)\int_0^{\frac\pi2}\frac{\theta^{2\lz-1}}{[(1-s)^2+4s(\sin\frac\theta2)^2)]^{\lz+1}}\,d\theta\\
&\gs&(1-s)\int_0^{\frac\pi2}\frac{\theta^{2\lz-1}}{[(1-s)^2+s\theta^2]^{\lz+1}}\,d\theta\\
&\gs&\frac1{(1-s)^{2\lz+1}}\int_0^{\frac\pi2}\frac{\theta^{2\lz-1}}{[1+(\frac{\sqrt s}{1-s}\theta)^2]^{\lz+1}}\,d\theta\\
&\gs&\int_0^{\frac\pi2\frac{\sqrt{K_1}}{1-K_1}}\frac{\beta^{2\lz-1}}{(1+\bz^2)^{\lz+1}}\,d\bz\gs1.
\end{eqnarray*}
Thus, by the inequality above and the fact that $(1-\tilde K_2)y\le y-z<y$, we conclude that
$$-\riz(y, z)\gs \frac1{y^{2\lz+1}}{\rm I}\gs \frac1{y^{2\lz}(y-z)},$$
and finish the proof of Proposition \ref{p-lower bdd riesz}.
%
%Now we show vii). As in Proposition 2.3 in \cite{dlwy}, for any $k>2$, we write
% $$\riz(x,kx)=\frac{-2\lz}{\pi}\frac1{x^{2\lz+1}}\dint_0^\pi\dfrac{(1-k\cos\theta)(\sin\theta)^{2\lz-1}}
%{(1+k^2-2k\cos\theta)^{\lz+1}}\,d\theta=:\frac{-2\lz}{\pi}\frac1{x^{2\lz+1}}{\rm I}.$$
%Then
%\begin{eqnarray*}
%{\rm I}&=&\int_0^{\pi}\frac\szlzm{(1+k^2-2k\cos\theta)^{\lz+1}}\,d\theta\\
%&\quad&+\lf[\left(\int_0^{\pi/2}+\int_{\pi/2}^{\pi}\right)\frac{-k\cos\theta}{(1+k^2-2k\cos\theta)^{\lz+1}}\szlzmx\r]\\
%&=&\int_0^{\pi}\frac\szlzm{(1+k^2-2k\cos\theta)^{\lz+1}}\,d\theta\\
%&\quad&-k\int_0^{\pi/2}\lf[\frac1{(1+k^2-2k\cos\theta)^{\lz+1}}-\frac1{(1+k^2+2k\cos\theta)^{\lz+1}}\r]\cos\theta\szlzmx\\
%&=:&{\rm I}_1-{\rm I}_2.
%\end{eqnarray*}
%Since $k>2$, we see that
%$$0<{\rm I}_1\le \int_0^\pi\frac\szlzm{(k-1)^{2\lz+2}}\,d\theta=\frac{\Gamma(\lz)\sqrt\pi}{\Gamma(\lz+1/2)}\frac1{(k-1)^{2\lz+2}}\ls
%\frac1{k^{2\lz+2}},$$
%and
%$$0<{\rm I}_2\le\frac{\Gamma(\lz)\sqrt\pi}{\Gamma(\lz+1/2)}\frac{k^2}{(k-1)^{2\lz+4}}\frac{\lz+1}{\lz+1/2}\ls
%\frac1{k^{2\lz+2}}.$$
%This implies
%$$|\riz(x, kx)|\ls \frac1{x^{2\lz+1}}|{\rm I}|\ls \frac1{x^{2\lz+1}}\lf[{\rm I}_1+{\rm I}_2\r]\ls \frac x{y^{2\lz+2}},$$
%and hence finishes the proof of vii).
\end{proof}

\section{An equivalent characterization of $\cmoz$}\label{s3}

In this section, we establish an equivalent characterization of $\cmoz$,
which is of independent interest. See also \cite{u78}.

\begin{thm}\label{t-cmo char}
Let $f\in\bmoz$. Then $f\in\cmoz$ if and only if $f$ satisfies the following three conditions:
\begin{itemize}
  \item [{\rm(i)}] $$\lim_{a\to0^+}\sup_{m_\lz(I)=a}M_\lz(f, I)=0,$$
  \item [{\rm(ii)}]$$\lim_{a\to\fz}\sup_{m_\lz(I)=a}M_\lz(f, I)=0,$$
  and
  \item [{\rm (iii)}]$$\lim_{R\to\fz}\sup_{I\subset[R,\,\infty)}M_\lz(f, I)=0.$$
%where $I+x:=\{y+x:\,\,y\in I\}$.
\end{itemize}
\end{thm}

\begin{proof}
Assume that $f\in \cmoz$.
If $f\in\cd$, then (i)-(iii) hold. In fact, (i) holds for $f$ since $f$ is uniformly continuous,
(ii) holds since $f\in\loz$, and (iii) holds
by the fact that $f$ is compactly supported.
If $f\in\cmoz\setminus\cd$, then for any given $\ez>0$, there exists $f_\ez\in\cd$ satisfying
(i)-(iii) and $\|f-f_\ez\|_\bmoz<\ez$. By the triangle inequality of $\bmoz$ norm, we see that (i)-(iii) hold for $f$.

Now we prove the converse. To this end, we assume that $f$ satisfies (i)-(iii).
  To prove that $f\in\cmoz$, it suffices to show
that there exists a positive constant $C_1$ depending only on $\lz$ such that,
for any $\ez\in(0,1)$, there exists $g_\ez\in\bmoz$ satisfying that
\begin{equation}\label{appro cmo-1}
\inf_{h\in\cd}\|g_\ez-h\|_\bmoz<C_1\ez
\end{equation}
and
\begin{equation}\label{appro com-2}
\|g_\ez-f\|_\bmoz<C_1\ez.
\end{equation}

We prove \eqref{appro cmo-1} and \eqref{appro com-2} by the following two steps.

{\bf Step I} We define an auxiliary function $\tilde g_\ez$ via a set of dyadic intervals $\mathcal I$ of $\rrp$.
In fact, by (i) and (ii), there exist $i_\ez, k_\ez\in\nn$ such that
\begin{equation}\label{i epsilon}
\sup\lf\{\mofi:\,\,m_\lz(I)\le 2^{-i_\ez+1}\r\}<\ez
\end{equation}
and
\begin{equation}\label{k epsilon}
\sup\lf\{\mofi:\,\,m_\lz(I)\ge 2^{k_\ez}\r\}<\ez.
\end{equation}
By (iii), there exists an integer $j_\ez>k_\ez$ such that
\begin{equation}\label{j epsilon}
\sup\lf\{\mofi:\,\, I\cap R_{j_\ez}=\emptyset\r\}<\ez,
\end{equation}
where $R_{j_\ez}$ is as in \eqref{i-gener}.
%In fact, by (iii) we can find $R>0$ such that $M_\lz(f,\,J)<\ez2^{-i_\ez-k_\ez-1}$, for each interval $J\subset[R,\,\infty)$ such that $m_\lz(J)=2^{k_\ez}$. Suppose that $I$ is an interval contained in $[R,\,+\infty)$ such that $m_\lz(I)<2^{k_\ez}$. If $m_\lz(I)\leq2^{-i_\ez}$, \eqref{i epsilon} implies that $\mofi<\ez$. Now assume that $m_\lz(I)>2^{-i_\ez}$. We choose an interval $J\subset[R,\,+\infty)$ satisfying that $m_\lz(J)=2^{k_\ez}$ and $I\subset J$, and write that
%\begin{eqnarray*}
%\frac{1}{m_\lz(I)}\int_I|f(x)-f_{I,\,\lz}|dm_\lz(x)
%&\leq&\frac{1}{m_\lz(I)}\int_I|f(x)-f_{J,\,\lz}|dm_\lz(x)+|f_{I,\,\lz}-f_{J,\,\lz}|\\
%&\leq&\frac{m_\lz(J)}{m_\lz(I)}\frac{1}{m_\lz(J)}\int_J|f(x)-f_{J,\,\lz}|dm_\lz(x)\\
%&\leq&2^{k_\ez+i_\ez+1}M_\lz(f,\,J)<\ez.
%\end{eqnarray*}
%By \eqref{k epsilon} $\mofi<\ez$ provided that $I\subset[R,\,+\infty)$ and $m_\lz(I)\geq2^{k_\ez}$.
%Thus the existence of $j_\ez\in\nn$, $j_\ez>k_\ez$, satisfying \eqref{j epsilon} is assured.

For the above $j_\ez$, we consider the dyadic intervals $R_{j_\ez}:=(0,2^{j_\ez}]$, $R_m\setminus R_{m-1}:=( 2^{m-1}, 2^m]$, ${m>j_\ez}$.
Next for $k=1,\ldots, 2^{j_\ez+i_\ez+2+\lfloor2\lz(j_\ez+1)\rfloor}$, we denote by
$$ I_k^{j_\ez}:= \lf((k-1)2^{-i_\ez-2-\lfloor2\lz(j_\ez+1)\rfloor}, k 2^{-i_\ez-2-\lfloor2\lz(j_\ez+1)\rfloor}\r]$$
the descendants of $R_{j_\ez}$. Here for any $\az\in\rr$, $\lfloor \az\rfloor$
means the largest integer $k$ such that $k\le \az$. And similarly, for each $m>j_\ez$ and $k=1,\ldots, 2^{j_\ez+i_\ez+1+\lfloor2\lz(j_\ez+1)\rfloor}$,
denote by
$$ I_k^{m}:= \lf(2^{m-1}+(k-1)2^{-i_\ez-2-\lfloor2\lz(j_\ez+1)\rfloor + m-j_\ez}, 2^{m-1}+k 2^{-i_\ez-2-\lfloor2\lz(j_\ez+1)\rfloor + m-j_\ez}\r]$$
the descendants of $R_{m}\setminus R_{m-1}$.

Then we list these dyadic descendants in order as follows:
$$\mathcal I:=\lf\{I^{j_\ez}_1,\ldots,\,I^{j_\ez}_{2^{j_\ez+i_\ez+2+\lfloor2\lz(j_\ez+1)\rfloor}},
\,I^{j_\ez+1}_1,\ldots, I^{j_\ez+1}_{2^{j_\ez+i_\ez+1+\lfloor2\lz(j_\ez+1)\rfloor}},
I^{j_\ez+2}_1,\,\ldots\r\}.$$

For each $x\in\rrp$, we define  $I_x$ as follows:  if $I\in\mathcal{I}$ and $x\in I$, then $I_x := I$. Observe that for each $x\in\mathbb R_+$, such $I_x$ exists and is unique.

%as follows.  Let $I_x$ be the dyadic interval that contains $x$
%5with length $2^{i_\ez-2-\lfloor2\lz(j_\ez+1)\rfloor}$ when $x\in R_{j_\ez}$,
%and with length $2^{i_\ez-2-\lfloor2\lz(j_\ez+1)\rfloor+m-j_\ez}$ when $x\in R_m\setminus R_{m-1}$ with $m>j_\ez$,
%where

We claim that
\begin{itemize}
  \item [(a)] Every dyadic interval $I$ in
$$\lf\{I^{j_\ez}_1,\cdots,\,I^{j_\ez}_{2^{j_\ez+i_\ez+2+\lfloor2\lz(j_\ez+1)\rfloor}},\,I^{j_\ez+1}_1,\,I^{j_\ez+1}_2\r\}$$
satisfies that $m_\lz(I)\le 2^{-i_\ez}$.

  \item [(b)] For any $m_\ez> j_\ez$ and $x\in R_{m_\ez}\setminus R_{m_\ez-1},$
\begin{equation}\label{ix length}
\lf|2^{j_\ez+\lfloor2\lz(j_\ez+1)\rfloor+2+i_\ez} I_x\r|= 2^{m_\ez}=|R_{m_\ez}|,
\end{equation}
\begin{equation}\label{ix inclusion}
2^{j_\ez+\lfloor2\lz(j_\ez+1)\rfloor+2+i_\ez} I_x\subset R_{m_\ez+1},
\end{equation}
moreover, if  $2\lz(m_\ez-j_\ez-2)\ge2$, then
\begin{equation}\label{ix meas lower bdd}
m_\lz(I_x)\ge 2^{m_\ez-i_\ez-j_\ez}.
\end{equation}
\end{itemize}

In fact, since $m_\lz(I_x)$ is non-decreasing with respect to $x$, to show (a), we only need to show that $m_\lz(I^{j_\ez+1}_2)\le 2^{-i_\ez}$. Observe that
$$I_2^{j_\ez+1}:=\lf(2^{j_\ez}+2^{-i_\ez-1-\lfloor2\lz(j_\ez+1)\rfloor}, 2^{j_\ez}+2^{-i_\ez-\lfloor2\lz(j_\ez+1)\rfloor}\r].$$
From this and the mean value theorem, it follows that there exists $\xi\in I_2^{j_\ez+1}$ such that
\begin{eqnarray}\label{ix-upper bdd}
m_\lz\lf(I_2^{j_\ez+1}\r)&=&\xi^{2\lz}2^{-i_\ez-1-\lfloor2\lz(j_\ez+1)\rfloor}\noz\\
&\le&2^{2\lz j_\ez}\lf[1+2^{-i_\ez-(2\lz+1)j_\ez-2\lz+1}\r]^{2\lz} 2^{-i_\ez-2\lz(j_\ez+1)}\noz\\
&=&\lf[\frac{1+2^{-i_\ez-(2\lz+1)j_\ez-2\lz+1}}2\r]^{2\lz} 2^{-i_\ez}\le 2^{-i_\ez}.
\end{eqnarray}
This implies (a) holds.

To show (b), we first observe that \eqref{ix length} is obvious.
Moreover,  since \eqref{ix inclusion} holds for
$$I^{m_\ez}_{2^{j_\ez+1+\lfloor2\lz(j_\ez+1)\rfloor+i_\ez}}:=\lf(2^{m_\ez}-2^{-i_\ez-2-\lfloor2\lz(j_\ez+1)\rfloor+m_\ez-j_\ez}, 2^{m_\ez}\r],$$
the last interval of $\mathcal I$ included in $R_{m_\ez}\setminus R_{m_\ez-1}$,
we also have that \eqref{ix inclusion} holds for any $I_x$ with $x\in R_{m_\ez}\setminus R_{m_\ez-1}.$

Finally, by the fact that $m(I_x)$ is non-decreasing in $x$, it suffices to
show \eqref{ix meas lower bdd} holds for $I_1^{m_\ez}$, the first dyadic interval of $\mathcal I$ included in
$ R_{m_\ez}\setminus R_{m_\ez-1}$. Observe that
$$I_1^{m_\ez}:=\lf(2^{m_\ez-1}, 2^{m_\ez-1}+2^{-i_\ez-2-\lfloor2\lz(j_\ez+1)\rfloor+m_\ez-j_\ez}\r]$$
and there exists $\xi\in I_1^{m_\ez}$ such that
$$m_\lz\lf(I_1^{m_\ez}\r)=\xi^{2\lz}2^{-i_\ez-2-\lfloor2\lz(j_\ez+1)\rfloor+m_\ez-j_\ez}
\ge 2^{2\lz(m_\ez-j_\ez-2)-2}2^{m_\ez-i_\ez-j_\ez}\ge 2^{m_\ez-i_\ez-j_\ez}$$
provided $2\lz(m_\ez-j_\ez-2)-2\ge0$. This implies \eqref{ix meas lower bdd}.

 Now for each $x\in\rrp$, let $\tilde g_\ez(x):=f_{I_x,\,\lz}$,
where $f_{I_x,\,\lz}$ is defined as in \eqref{average}. Then from
(ii), there exists an integer $m_\ez>j_\ez$ such that
\begin{equation}\label{m epsilon}
\sup\lf\{\lf|\tilde g_\ez(x)-\tilde g_\ez(y)\r|:\,\,x,\,y\in R_{m_\ez}\setminus R_{m_\ez-1} \r\}<\ez.
\end{equation}
To see this, by (ii), let $m_\ez >j_\ez+k_\ez+i_\ez$ be large enough such that when $m_\lz(I)\ge 2^{m_\ez-i_\ez-j_\ez}$,
\begin{equation}\label{mofi epsilon}
\mofi<\{C_2[j_\ez+2\lz(j_\ez+1)+2+i_\ez]\}^{-1}\ez
\end{equation}
for some positive constant $C_2>2^{3(2\lz+1)+2}$.

%We see that
%\begin{equation}\label{R-m-epsilon-inclusion}
%R_{m_\ez+1}\subset8\cdot2^{j_\ez+\lfloor2\lz(j_\ez+1)\rfloor+2+i_\ez}I_x.
%\end{equation}
%In fact, if \eqref{R-m-epsilon-inclusion} holds for
%$$I_1^{m_\ez}=\lf(2^{m_\ez-1}, 2^{m_\ez-1}+2^{-i_\ez-2-\lfloor2\lz(j_\ez+1)\rfloor+m_\ez-j_\ez}\r],$$
%the first interval of $\mathcal I$ included in $R_{m_\ez}\setminus R_{m_\ez-1}$,
%we also have that \eqref{R-m-epsilon-inclusion} holds for any $I_x$ with $x\in R_{m_\ez}\setminus R_{m_\ez-1}.$
%To this end,
%$$I_1^{m_\ez}:=I(2^{m_\ez-1}+2^{-i_\ez-3-\lfloor2\lz(j_\ez+1)\rfloor+m_\ez-j_\ez},\,2^{-i_\ez-3-\lfloor2\lz(j_\ez+1)\rfloor+m_\ez-j_\ez}),$$
%thus,
%\begin{eqnarray*}
%8\cdot2^{j_\ez+\lfloor2\lz(j_\ez+1)\rfloor+2+i_\ez}I_1^{m_\ez}
%&=&I(2^{m_\ez-1}+2^{-i_\ez-3-\lfloor2\lz(j_\ez+1)\rfloor+m_\ez-j_\ez},\,2^{m_\ez+2})\\
%&=&\lf(0,\,2^{m_\ez+2}+2^{m_\ez-1}+2^{-i_\ez-3-\lfloor2\lz(j_\ez+1)\rfloor+m_\ez-j_\ez}\r].
%\end{eqnarray*}
By \eqref{ix length} and \eqref{ix inclusion}, we see that
$$2^{j_\ez+\lfloor2\lz(j_\ez+1)\rfloor+2+i_\ez}I_x\subset R_{m_\ez+1}\subset4\cdot2^{j_\ez+\lfloor2\lz(j_\ez+1)\rfloor+2+i_\ez}I_x.$$
This together with \eqref{reverse doubl} and \eqref{mofi epsilon} implies that
\begin{eqnarray}\label{m epsi minus-i}
\lf|f_{2^{j_\ez+\lfloor2\lz(j_\ez+1)\rfloor+2+i_\ez}I_x,\,\lz}-f_{R_{m_\ez+1},\,\lz}\r|
&\leq&\frac{m_\lz(R_{m_\ez+1})}{m_\lz(2^{j_\ez+\lfloor2\lz(j_\ez+1)\rfloor+2+i_\ez}I_x)}M_\lz(f,\,R_{m_\ez+1})\noz\\
&<&2^{2(2\lz+1)}\frac{\ez}{C_2[j_\ez+2\lz(j_\ez+1)+2+i_\ez]}\noz\\
&<&\ez/8.
\end{eqnarray}
Similarly, observe that $R_{m_\ez+1}\subset 8(R_{m_\ez}\setminus R_{m_\ez-1})$. Thus by \eqref{reverse doubl},
%In fact,
%$$R_{m_\ez}\setminus R_{m_\ez-1}:=(2^{m_\ez-1},\,2^{m_\ez}]=I(2^{m_\ez-1}+2^{m_\ez-2},\,2^{m_\ez-2}),$$
%and
%$$8(R_{m_\ez}\setminus R_{m_\ez-1})=I(2^{m_\ez-1}+2^{m_\ez-2},\,2^{m_\ez+1})=(0,\,2^{m_\ez+1}+2^{m_\ez-1}+2^{m_\ez-2}].$$

\begin{eqnarray}\label{m epsi minus-ii}
\lf|f_{R_{m_\ez+1}}-f_{R_{m_\ez}\setminus R_{m_\ez-1}}\r|
&\leq&\frac{m_\lz(R_{m_\ez+1})}{m_\lz(R_{m_\ez}\setminus R_{m_\ez-1})}M_\lz(f,\,R_{m_\ez+1})\noz\\
&<&2^{3(2\lz+1)}\frac{\ez}{C_2[j_\ez+2\lz(j_\ez+1)+2+i_\ez]}\noz\\
&<&\ez/8.
\end{eqnarray}
By \eqref{mofi epsilon}, \eqref{m epsi minus-i}, \eqref{m epsi minus-ii} and \eqref{ix meas lower bdd}, we conclude that for any $I_x$ with $x\in R_{m_\ez}\setminus R_{m_\ez-1}$,
\begin{eqnarray}\label{m epsi minus}
&&\lf|f_{I_x,\,\lz}-f_{R_{m_\ez}\setminus R_{m_\ez-1},\,\lz}\r|\nonumber\\
&&\quad\le \lf|f_{2^{j_\ez+\lfloor2\lz(j_\ez+1)\rfloor +2+i_\ez}I_x,\,\lz}-f_{R_{m_\ez}\setminus R_{m_\ez-1},\,\lz}\r|+
\sum_{j=0}^{{j_\ez+\lfloor2\lz(j_\ez+1)\rfloor+1+i_\ez}}\lf|f_{2^jI_x,\,\lz}-f_{2^{j+1}I_x,\,\lz}\r|\nonumber\\
&&\quad\le\lf|f_{2^{j_\ez+\lfloor2\lz(j_\ez+1)\rfloor+2+i_\ez} I_x,\,\lz}-f_{R_{m_\ez+1},\,\lz}\r|
+\lf|f_{R_{m_\ez+1,\,\lz}}-f_{R_{m_\ez}\setminus R_{m_\ez-1},\,\lz}\r|\nonumber\\
&&\quad\quad+\sum_{j=0}^{{j_\ez+\lfloor2\lz(j_\ez+1)\rfloor+1+i_\ez}}2^{2\lz+1}\frac\ez{C_2[j_\ez+2\lz(j_\ez+1)+2+i_\ez]}\nonumber\\
&&\quad<\ez/8+\ez/8+\frac{2^{2\lz+1}}{C_2}\ez\noz\\
&&\quad<\ez/2.
\end{eqnarray}
So for any $I_x$, $I_y$ with $x, y\in R_{m_\ez}\setminus R_{m_\ez-1},$
$$\lf|f_{I_x,\,\lz}-f_{I_y,\,\lz}\r|\le \lf|f_{I_x,\,\lz}-f_{R_{m_\ez}\setminus R_{m_\ez-1},\,\lz}\r|
+\lf|f_{R_{m_\ez}\setminus R_{m_\ez-1},\,\lz}-f_{I_y,\,\lz}\r|<\ez.$$
This shows \eqref{m epsilon}.

{\bf Step II} Define $g_\ez(x):=\tilde g_\ez(x)$ when $x\in R_{m_\ez}$ and $g_\ez(x):=f_{R_{m_\ez}\setminus R_{m_\ez-1},\,\lz}$
when $x\in \rrp\setminus R_{m_\ez}$. Before proving \eqref{appro cmo-1} and \eqref{appro com-2}, we
first claim that there exists a positive constant $C_3$ such that
if $\bar{I_x}\cap\bar{I_y}\not=\emptyset$ or $x,\,y\in \rrp\setminus R_{m_\ez-1}$, then
\begin{equation}\label{g epsilon func}
|g_\ez(x)-g_\ez(y)|<C_3\ez.
\end{equation}

In fact,  assume that $x<y$. We first show that if  $x,\,y\in \rrp\setminus R_{m_\ez-1}$, then \eqref{g epsilon func} holds.
Firstly, if $x,\,y\in \rrp\setminus R_{m_\ez}$, then
$$g_\ez(x)=g_\ez(y)=f_{R_{m_\ez}\setminus R_{m_\ez-1},\,\lz}$$
 and \eqref{g epsilon func} holds.
 Secondly, if $x,\,y\in R_{m_\ez}\setminus R_{m_\ez-1}$, then from \eqref{m epsilon}, we deduce that
 $$\lf|g_\ez(x)-g_\ez(y)\r|=\lf|\tilde g_\ez(x)-\tilde g_\ez(y)\r|<\ez.$$
Thirdly, if $x\in  R_{m_\ez}\setminus R_{m_\ez-1}$ and $y\in \rrp\setminus R_{m_\ez}$, then from \eqref{m epsi minus},
it follows that
$$\lf|g_\ez(x)-g_\ez(y)\r|=\lf|\tilde g_\ez(x)-f_{R_{m_\ez}\setminus R_{m_\ez-1},\,\lz}\r|<\ez/2.$$

Now we show if $\bar I_x\cap \bar I_y\not=\emptyset$, then \eqref{g epsilon func} holds. In fact, assume that $I_x\not= I_y$
and define $I:=I_x\cup I_y$. Observe that by the choice of $I_x$,
$|I_y|/2\le|I_x|\le 2|I_y|$ if $\bar{I_x}\cap\bar{I_y}\not=\emptyset$.
If  $x,y\in R_{j_\ez}$ and $I_x,\,I_y\in\{I^{j_\ez}_1,\cdots,\,I^{j_\ez}_{2^{j_\ez+i_\ez+2+\lfloor2\lz(j_\ez+1)\rfloor}}\}$,
by \eqref{ix-upper bdd},
we have that $m_\lz(I)\le m_\lz(I_1^{j_\ez+1})\le 2^{-i_\ez+1}$. And from \eqref{i epsilon} and \eqref{reverse doubl}, it follows that
$$\lf|g_\ez(x)-g_\ez(y)\r|\le \lf|\tilde g_\ez(x)-f_{I,\,\lz}\r|+\lf|f_{I,\,\lz}-\tilde g_\ez(y)\r|\ls M_\lz(f, I)\ls \ez.$$
Similarly, if $I_x=I^{j_\ez}_{2^{j_\ez+i_\ez+2+\lfloor2\lz(j_\ez+1)\rfloor}}$,
$I_y=I^{j_\ez+1}_1$ or $I_x=I^{j_\ez+1}_1$ and $I_y=I^{j_\ez+1}_2$, then
arguing as in \eqref{ix-upper bdd}, we see that $m_\lz(I)\le 2^{-i_\ez+1}$. By \eqref{i epsilon} and \eqref{reverse doubl} again,
$$\lf|g_\ez(x)-g_\ez(y)\r|\le \lf|\tilde g_\ez(x)-f_{I,\,\lz}\r|+\lf|f_{I,\,\lz}-\tilde g_\ez(y)\r|\ls M_\lz(f, I)\ls \ez.$$
Finally, if $I_x\notin \{I^{j_\ez}_1,\cdots,\,I^{j_\ez}_{2^{j_\ez+i_\ez+2+\lfloor2\lz(j_\ez+1)\rfloor}},\,I^{j_\ez+1}_1\}$ and $y>x$, then
$I\cap R_{j_\ez}=\emptyset$. It follows from \eqref{j epsilon} and \eqref{reverse doubl} that
$$\lf|g_\ez(x)-g_\ez(y)\r|\le \lf|\tilde g_\ez(x)-f_{I,\,\lz}\r|+\lf|f_{I,\,\lz}-\tilde g_\ez(y)\r|\ls M_\lz(f, I)\ls \ez.$$
Combining these cases, \eqref{g epsilon func} holds.

The function $g_\ez$ satisfies \eqref{appro cmo-1}. In fact, let
$$\tilde h_\ez(x):=g_\ez(x)-f_{R_{m_\ez}\setminus R_{m_\ez-1},\,\lz}.$$
Then by the definition of $g_\ez$, we see that
$$\tilde h_\ez(x)=0\,\,{\rm for\, \,any}\,\,x\in \rrp\setminus R_{m_\ez},\,\,\lf\|\tilde h_\ez-g_\ez\r\|_\bmoz=0.$$
Moreover, if $\bar{I_x}\cap\bar{I_y}\not=\emptyset$ or $x,\,y\in \rrp\setminus R_{m_\ez-1}$, then from \eqref{g epsilon func},
it follows that
$$\lf|\tilde h_\ez(x)-\tilde h_\ez(y)\r|=|g_\ez(x)-g_\ez(y)|<C_3\ez.$$
Observe that $\supp(\tilde h_\ez)\subset R_{m_\ez}$ and there exists a function $h_\ez\in C_c(\rrp)$
such that for any $x\in\rrp$,
$$\lf|\tilde h_\ez(x)-h_\ez(x)\r|<C_3\ez.$$
Then let $\omega\in C_c(\rr)$ be a positive valued function with
$\int_\rr \omega(x)\,dx=1$ and $\omega_t(x):=\frac1t\omega(\frac xt)$ for any $t\in\rrp$ and $x\in\rr$.
Then we see that
$\omega_t\ast h_\ez(x)\to h_\ez(x)$ uniformly for $x\in\rrp$ as $t\to0^+$, which yields the following inequality:
\begin{align*}
\lf\|\omega_t\ast h_\ez-g_\ez\r\|_\bmoz&\le \lf\|\omega_t\ast h_\ez-h_\ez\r\|_\bmoz+\lf\|h_\ez-\wz h_\ez\r\|_\bmoz\\
&\quad+\lf\|\wz h_\ez-g_\ez\r\|_\bmoz\\
&\ls \lf\|\omega_t\ast h_\ez-h_\ez\r\|_\linz+\ez.
\end{align*}
Hence, by letting $t\to 0^+$ we get that \eqref{appro cmo-1} holds.

Now we show \eqref{appro com-2}. From the definitions of $i_\ez$ and $j_\ez$, we deduce that for any $x\in R_{m_\ez}$,
\begin{equation}\label{rm epsilon}
\int_{I_x}|f(y)-g_\ez(y)|\,y^{2\lz}dy\ls \ez m_\lz(I_x).
\end{equation}
In fact,
\begin{equation*}
\int_{I_x}|f(y)-g_\ez(y)|\,y^{2\lz}dy=\int_{I_x}\lf|f(y)-\tilde g_\ez(y)\r|\,y^{2\lz}dy=\int_{I_x}\lf|f(y)-f_{I_x,\,\lz}\r|\,y^{2\lz}dy.
\end{equation*}
If $I_x\cap R_{j_\ez}=\emptyset$, then by \eqref{j epsilon}, \eqref{rm epsilon} holds.
If $I_x\cap R_{j_\ez}\not=\emptyset$, i. e.,
$I_x\in\{I^{j_\ez}_1,\cdots\,I^{j_\ez}_{2^{j_\ez+i_\ez+2+\lfloor2\lz(j_\ez+1)\rfloor}}\}$, then $m_\lz(I_x)\le 2^{-i_\ez}$.
From  this fact and \eqref{i epsilon}, \eqref{rm epsilon} follows.

Let $I$ be an arbitrary interval in $\rrp$. To show \eqref{appro com-2}, we only need to prove that
\begin{equation}\label{appro g-ez}
M_\lz(f-g_\ez, I)\ls \ez.
\end{equation}
To this end, we consider the following four cases:

Case i) $I\subset R_{m_\ez}$ and $\max\{|I_x|:\,\, I_x\cap I\not=\emptyset\}>4|I|$. In this case, the cardinality of
the set $\{I_x:\,\, I_x\cap I\not=\emptyset\}$ is at most 2 and hence,
$\bar I_{x_i}\cap \bar I_{x_j}\not=\emptyset$ if $I_{x_i}\cap I\not=\emptyset$ and $I_{x_j}\cap I\not=\emptyset$.
By \eqref{g epsilon func}, we have that
\begin{eqnarray*}
M_\lz(g_\ez, I)&&\le \frac1{m_\lz(I)}\sum_{i:\, I_{x_i}\cap I\not=\emptyset}\int_{I_{x_i}\cap I}\frac1{m_\lz(I)}
\sum_{j:\, I_{x_j}\cap I\not=\emptyset}\int_{I_{x_j}\cap I}
\lf|g_\ez(x)-g_\ez(y)\r|y^{2\lz}\,dy\,\xtz\ls\ez.
\end{eqnarray*}
Moreover, if $I\cap R_{j_\ez}\not=\emptyset$, then
$m_\lz(I)\le m_\lz(I^{j_\ez+1}_1)\le 2^{-i_\ez}$.  By \eqref{i epsilon}, we see that
$\mofi<\ez$ and so
$$M_\lz(f-g_\ez, I)\le \mofi+M_\lz(g_\ez, I)<\ez+M_\lz(g_\ez, I)\ls \ez.$$
If $I\cap R_{j_\ez}=\emptyset$, then by \eqref{j epsilon}, we also
see that $\mofi<\ez$ and $M_\lz(f-g_\ez, I)\ls \ez$.

Case ii) $I\subset R_{m_\ez}$ and $\max\{|I_x|:\,\, I_x\cap I\not=\emptyset\}\leq4|I|$.
In this case, from \eqref{reverse doubl}, it follows that
$$\sum_{i:\, I_{x_i}\cap I\not=\emptyset}m_\lz(I_{x_i})\sim m_\lz(I).$$
Since $I\subset R_{m_\ez}$, then $x\in R_{m_\ez}$ if $I_x\cap I\not=\emptyset$.
By this and \eqref{rm epsilon}, we see that
$$M_\lz(f-g_\ez, I)\ls\frac1{m_\lz(I)}\sum_{i:\, I_{x_i}\cap I\not=\emptyset}
\int_{I_{x_i}}\lf|f(y)-g_\ez(y)\r|\,y^{2\lz}dy\ls\frac1{m_\lz(I)}\sum_{i:\, I_{x_i}\cap I\not=\emptyset} m_\lz(I_{x_i})\ez\ls\ez.$$
Thus, \eqref{appro g-ez} holds in this case.

%Case a) $I\subset R_{m_\ez-1}$. If $I_x\cap I\not=\emptyset$, then $I_x\subset R_{m_\ez-1}$.
%This implies that
%$$M_\lz(f-g_\ez, I)\ls \frac1{m_\lz(I)}\sum_i\int_{I_i}|f(y)-f_{I_i}|\,y^{2\lz}dy\ls\frac\ez{m_\lz(I)}\sum_im_\lz(I_i)\ls\ez.$$
%Case b) $I\subset (R_{m_\ez}\setminus R_{m_\ez-1})$. It follows from \eqref{m epsilon} that
%\begin{eqnarray*}
%M_\lz(f-g_\ez, I)&\ls& \frac1{m_\lz(I)}\sum_i\int_{I_i}\lf|f(y)-f_{R_{m_\ez}\setminus R_{m_\ez-1},\,\lz}\r|\,y^{2\lz}dy\\
%&\ls&\frac1{m_\lz(I)}\sum_i\lf[\int_{I_i}\lf|f(y)-f_{I_i,\,\lz}\r|\,y^{2\lz}dy
%+m_\lz(I_i)\lf|f_{I_i,\,\lz}-f_{R_{m_\ez}\setminus R_{m_\ez-1},\,\lz}\r|\r]\\
%&\ls&\frac1{m_\lz(I)}\sum_im_\lz(I_i)\lf[\ez+\frac{\ez}{j_\ez-i_\ez}(j_\ez-i_\ez)\r]\ls\ez.
%\end{eqnarray*}
%
%Case c) $I\cap R_{m_\ez-1}\not=\emptyset$ and $I\cap \rrp\setminus R_{m_\ez-1}\not=\emptyset$. From a) and b).

Case iii) $I\subset (\rrp\setminus R_{m_\ez-1})$. In this case, $I\cap R_{j_\ez}=\emptyset$.
By \eqref{j epsilon}, we see that $\mofi<\ez$. Similar to Case i), it then suffices to estimate $M_\lz(g_\ez, I)$.
However, by \eqref{g epsilon func}, $M_\lz(g_\ez, I)\ls\ez$. Thus, \eqref{appro g-ez} holds.

%Case a) $I\subset (\rrp\setminus R_{m_\ez})$. In this case, by the fact that $g_\ez(x)\equiv f_{R_{m_\ez}\setminus R_{m_\ez-1},\,\lz}$
%on $I\subset (\rrp\setminus R_{m_\ez})$, we see that $M_\lz(g_\ez, I)=0$.
%
%Case b) $I\subset (R_{m_\ez}\setminus R_{m_\ez-1})$. In this case, $M_\lz(g_\ez, I)=M_\lz(\tilde g_\ez, I)<\ez$.
%
%Case c) $I\cap (\rrp\setminus R_{m_\ez})\not=\emptyset$ and $I\cap R_{m_\ez}\not=\emptyset$. If
%$\max\{|I_x|:\,\, I_x\cap I\not=\emptyset\}>4|I|$, then the set $\{|I_x|:\,\, I_x\cap I\not=\emptyset\}$ is at most 3.
%$$M_\lz(g_\ez, I)\le \frac1{m_\lz(I)}\sum_i\int_{I_i\cap I}\frac1{m_\lz(I)}\sum_j\int_{I_j\cap I}\lf|g_\ez(x)-g_\ez(y)\r|y^{2\lz}\,dy\xtz\ls\ez.$$
%
%If $\max\{|I_x|:\,\, I_x\cap I\not=\emptyset\}\le4|I|$,
%\begin{eqnarray*}
%M_\lz(g_\ez, I)&=&\frac1{m_\lz(I)}\lf[\int_{I\setminus R_{m_\ez}}+\int_{I\cap R_{m_\ez}}\r]\lf|g_\ez(x)-(g_\ez)_{I,\,\lz}\r|\xtz\\
%&\le &\frac1{m_\lz(I)}\lf[\int_{I\setminus R_{m_\ez}}\frac1{m_\lz(I)}
%\int_{I \cap R_{m_\ez}}|f_{R_{m_\ez}\setminus R_{m_\ez-1},\,\lz}-\tilde g_\ez(y)|y^{2\lz}\,dy\xtz\r.\\
%&\quad&+ \int_{I \cap R_{m_\ez}}\frac1{m_\lz(I)}\lf[\int_{I \cap R_{m_\ez}}|\tilde g_\ez(x)-g_\ez(y)|y^{2\lz}\,dy\r.\\
%&\quad&\lf.\lf.+\int_{I \setminus R_{m_\ez}}|\tilde g_\ez(x)-f_{R_{m_\ez}\setminus R_{m_\ez-1},\,\lz}|y^{2\lz}\,dy\r]\xtz\r]\ls\ez.
%\end{eqnarray*}

Case iv) $I\cap(\rrp\setminus R_{m_\ez})\neq\emptyset$ and $I\cap R_{m_\ez-1}\not=\emptyset$.
Let $p_I$ be the smallest integer such that $I\subset R_{p_I}$.
Then by \eqref{reverse doubl},
$$M_\lz(f-g_\ez, I)\ls M_\lz(f-g_\ez, R_{p_I}).$$
%Since $m_\ez> k_\ez$, for any $q\ge m_\ez$,
%\begin{equation*}
%\lf|f_{R_q,\,\lz}-f_{R_{q-1},\,\lz}\r|\ls\ez.
%\end{equation*}
Moreover,
\begin{eqnarray*}
&&M_\lz\lf(f-g_\ez, R_{p_I}\r)m_\lz\lf(R_{p_I}\r)\\
&&\quad\ls\int_{R_{p_I}}\lf|(f-g_\ez)(x)-(f-g_\ez)_{R_{p_I}\setminus R_{m_\ez},\,\lz}\r|\,\xtz\\
&&\quad\ls\int_{R_{p_I}}\lf|f(x)-f_{R_{p_I}\setminus R_{m_\ez},\,\lz}\r|\,\xtz+\int_{R_{p_I}}\lf|g_\ez(x)-
(g_\ez)_{R_{p_I}\setminus R_{m_\ez},\,\lz}\r|\,\xtz.
\end{eqnarray*}
On the one hand, observe that $m_\lz(R_{p_I})\ge m_\lz(R_{m_\ez})\ge 2^{k_\ez}$. By  this,  \eqref{reverse doubl} and \eqref{k epsilon},
we have that
\begin{eqnarray*}
&&\int_{R_{p_I}}\lf|f(x)-f_{R_{p_I}\setminus R_{m_\ez},\,\lz}\r|\,\xtz\\
&&\quad\le\int_{R_{p_I}}\lf|f(x)-f_{R_{p_I},\,\lz}\r|\,\xtz+\lf|f_{R_{p_I},\,\lz}-f_{R_{p_I}\setminus R_{m_\ez},\,\lz}\r|m_\lz(R_{p_I})\\
&&\quad\ls
\int_{R_{p_I}}\lf|f(x)-f_{R_{p_I},\,\lz}\r|\,\xtz
\ls \ez m_\lz(R_{p_I}).
\end{eqnarray*}

On the other hand, it is obvious that
$$\sum_{i:\,I_{x_i}\in \mathcal I,\,I_{x_i}\subset R_{m_\ez}}m_\lz(I_{x_i})= m_\lz(R_{m_\ez}).$$
From this, the fact that $g_\ez(x)=g_\ez(y)$ for any $x,\,y\in(\rrp\setminus R_{m_\ez})$,
\eqref{k epsilon} and \eqref{rm epsilon}, we deduce that
\begin{eqnarray*}
&&\int_{R_{p_I}}\lf|g_\ez(x)-(g_\ez)_{R_{p_I}\setminus R_{m_\ez},\,\lz}\r|\,\xtz\\
&&\quad\le \frac1{m_\lz(R_{p_I}\setminus R_{m_\ez})}\int_{R_{p_I}}\int_{R_{p_I}\setminus R_{m_\ez}}
\lf|g_\ez(x)-g_\ez(y)\r|y^{2\lz}\,dy\,\xtz\\
&&\quad=\frac1{m_\lz(R_{p_I}\setminus R_{m_\ez})}\int_{R_{m_\ez}}\int_{R_{p_I}\setminus R_{m_\ez}}\lf|g_\ez(x)-g_\ez(y)\r| y^{2\lz}\,dy\,\xtz\\
&&\quad\le \frac1{m_\lz(R_{p_I}\setminus R_{m_\ez})}\int_{R_{m_\ez}}\int_{R_{p_I}\setminus R_{m_\ez}}
\lf[\lf|g_\ez(x)-f(x)\r|+\lf|f(x)-f_{R_{m_\ez}\setminus R_{m_\ez-1},\,\lz}\r|\r]y^{2\lz}\,dy\,\xtz\\
&&\quad\le \frac1{m_\lz(R_{p_I}\setminus R_{m_\ez})}\int_{R_{p_I}\setminus R_{m_\ez}}\sum_{i:\,\,I_{x_i}\in \mathcal I,\,I_{x_i}\subset R_{m_\ez}}
\int_{I_{x_i}}
\lf|g_\ez(x)-f(x)\r|\xtz\, y^{2\lz}\,dy\\
&&\quad\quad+\int_{R_{m_\ez}}\lf[\lf|f(x)-f_{R_{m_\ez},\,\lz}\r|
+\lf|f_{R_{m_\ez},\,\lz}-f_{R_{m_\ez}\setminus R_{m_\ez-1},\,\lz}\r|\r]\,\xtz\\
&&\quad\ls \frac1{m_\lz(R_{p_I}\setminus R_{m_\ez})}\int_{R_{p_I}\setminus R_{m_\ez}}\ez\sum_{i:\,I_{x_i}\in \mathcal I,\,I_{x_i}\subset R_{m_\ez}}
m_\lz(I_{x_i})\, y^{2\lz}\,dy+\int_{R_{m_\ez}}\lf|f(x)-f_{R_{m_\ez},\,\lz}\r|\,\xtz\\
&&\quad\ls \ez m_\lz(R_{m_\ez})\\
&&\quad\ls \ez m_\lz(R_{p_I}).
\end{eqnarray*}

%\begin{eqnarray*}
%M_\lz\lf(f-g_\ez, R_{p_I}\r)m_\lz\lf(R_{p_I}\r)&\le&\lf[\int_{R_{p_I}\setminus R_{m_\ez}}+\int_{R_{m_\ez}}\r]
%\lf|f(x)-f_{R_{p_I},\,\lz}\r|\xtz\\
%&\quad&+\int_{R_{p_I}}\lf|g_\ez(x)-(g_\ez)_{R_{p_I},\,\lz}\r|\xtz
%\end{eqnarray*}

%Observe that $m_\lz(R_\ez)> 2^{k_\ez}$. By \eqref{k epsilon},
%\begin{eqnarray*}
%\int_{R_{m_\ez}}\lf|f(x)-f_{R_{p_I},\,\lz}\r|\xtz&\le& \int_{R_{m_\ez,\,\lz}}
%\lf|f(x)-f_{R_{m_\ez,\,\lz}}\r|\xtz+\lf|f_{R_{m_\ez,\,\lz}}-f_{R_{p_I},\,\lz}\r|m_\lz(R_{m_\ez})\\
%&\ls&(p_I-m_\ez)\ez m_\lz(R_{m_\ez})
%\end{eqnarray*}
%
%On the other hand,
%\begin{eqnarray*}
%\int_{R_{p_I}}\lf|g_\ez(x)-(g_\ez)_{R_{p_I},\,\lz}\r|\xtz
%&=&\lf[\int_{R_{p_I}\setminus R_{m_\ez}}+\int_{R_{m_\ez}}\r]\lf|g_\ez(x)-(g_\ez)_{R_{p_I},\,\lz}\r|\xtz\\
%&\le&\int_{R_{p_I}\setminus R_{m_\ez}}\frac1{m_\lz(R_{p_I})}\int_{R_{m_\ez}}\lf|f_{R_{m_\ez,\,\lz}}-\tilde g_\ez(y)\r|y^{2\lz}\,dy\xtz\\
%&\ls&\ez m_\lz(R_{p_I}\setminus R_{m_\ez}).
%\end{eqnarray*}
%and by \eqref{m epsilon} and \eqref{m epsi minus},
%\begin{eqnarray*}
%&&\int_{R_{m_\ez}}\lf|g_\ez(x)-(g_\ez)_{R_{p_I},\,\lz}\r|\xtz\\
%&&\quad=\int_{R_{m_\ez}}\frac1{m_\lz(R_{p_I})}\lf|\lf[\int_{R_{p_I}\setminus R_{m_\ez}}+\int_{R_{m_\ez}}\r]g_\ez(x)-g_\ez(y)y^{2\lz}\,dy\r|\,\xtz\\
%&&\quad\le \int_{R_{m_\ez}}\frac1{m_\lz(R_{p_I})}\lf[\int_{R_{p_I}\setminus R_{m_\ez}}\lf|g_\ez(x)
%-f_{R_{m_\ez}\setminus R_{m_\ez-1},\,\lz}\r|y^{2\lz}\,dy\r.\\
%&&\quad\quad\lf.+\int_{R_{m_\ez}}\lf|g_\ez(x)-g_\ez(y)\r|y^{2\lz}\,dy\r]\,\xtz\\
%&&\quad\ls \ez m_\lz(R_{m_\ez}).
%\end{eqnarray*}

This implies \eqref{appro com-2} and finishes the proof of Theorem \ref{t-cmo char}.
\end{proof}

\section{The Fr\'{e}chet-Kolmogorov theorem in the Bessel setting}\label{s4}

In this section, we provide a version of Fr\'{e}chet-Kolmogorov theorem in the Bessel setting,  stating
a necessary and sufficient condition for a subset of $L^p$ to be relatively compact, which is useful in the proof of Theorem \ref{t-riesz compact}.
 For the original  Fr\'{e}chet-Kolmogorov theorem, we refer
the readers to  Yosida \cite{yo}. See also
\cite{cc, hh}.

%In this section, we establish
%the Fr\'{e}chet-Kolmogorov theorem in $L^p(\mathbb{R}_+,\,dm_\lz)$,

We first recall that a metric space $(\cx, d)$ is totally bounded if for every $\ez>0$, there exists a finite number of open balls  of radius
$\ez$ whose union is the space $\cx$, and a metric space $(\cx, d)$ is compact
if and only if it is complete and totally bounded; see, for example,
\cite{cc}.
%
%The idea of proving Theorem \ref{nece-suffi-comp} is adopt from \cite{cc, hh}.
%In particular, the following lemma is essential for sufficient condition in the proof of Theorem \ref{nece-suffi-comp}.

\begin{lem}{\rm (\cite{hh})}\label{l-total bdd}
Let $(\cx, d)$ be a metric space. Suppose that for every $\epsilon>0$, there exists some
$\delta>0$, a metric space $({\mathcal W},\tilde d)$ and a mapping $\Phi$: $\cx\rightarrow {\mathcal W}$ such that $\Phi(\cx)$
is totally bounded, and whenever $x, y\in \cx$ are such that
$\tilde d(\Phi(x), \Phi(y))<\delta$, then $d(x, y)<\epsilon$. Then $X$ is totally bounded.
\end{lem}

The main result of this section is as follows.

\begin{thm}\label{t-fre kol}
For  $1<p<\infty$, a subset $\cf$ of $L^p(\mathbb R_+,\,dm_\lz)$ is totally bounded (or relatively compact) if and only if
the following statements hold:

{\rm(a)}\ $\cf$ is uniformly bounded, i.e., $\sup_{f\in\cf}\|f\|_{L^p(\mathbb R_+,\,dm_\lz)}<\infty$;

{\rm(b)}\ $\cf$ uniformly vanishes at infinity, i.e., for every $\epsilon>0$, there exists some positive constant $M$
 such that for every $f\in\cf$,
$$\int_M^\infty|f(x)|^px^{2\lambda}\,dx<\epsilon^p;$$

{\rm(c)}\  $\cf$ is uniformly equicontinuous, i.e., for every $\epsilon>0$, there exists some positive constant $\rho$,
such that for every $f\in\cf$ and $y\in \mathbb R_+$ with $y<\rho$,
$$\int_0^\infty|f(x+y)-f(x)|^px^{2\lambda}\,dx<\epsilon^p.$$
\end{thm}

\begin{proof}
Assume that $\cf\subset L^p(\mathbb R_+,\,dm_\lz)$ satisfies the three conditions.
By Lemma \ref{l-total bdd}, to show $\mathcal F$ is totally bounded, it suffices to prove that
for any $\ez>0$, there exists a mapping $\Phi$ on $\lpz$ such that $\Phi(\mathcal F)$ is totally bounded
and that
\begin{equation}\label{totally bdd}
\|f-g\|_{L^p(\mathbb R_+,\,dm_\lz)}<\epsilon
\end{equation}
for any $f,\,g\in\cf$ such that
\begin{equation}\label{phi lpz regularity}
\|\Phi(f)-\Phi(g)\|_{L^p(\mathbb R_+,\,dm_\lz)}<\epsilon/2.
\end{equation}

To this end, given $\epsilon>0$,
pick $M$ as in the condition ${\rm(b)}$, such that
\begin{equation}\label{f-vani-bd}
\sup_{f\in\cf}\|f-f\chi_{(0,\,M)}\|_{L^p(\mathbb R_+,\,dm_\lz)}<\frac{\epsilon}{12}.
\end{equation}
Let  $\rho$  be as in condition ${\rm (c)}$ such that
\begin{equation}\label{f-equaconti-bd}
\sup_{y\in (0,\,\rho]}\lf(\sup_{f\in\cf}\|f(\cdot+y)-f(\cdot)\|_{L^p(\mathbb R_+,\,dm_\lz)}\r)<\frac{\epsilon}{[2(2\lz+1)]^{1/p}12}.
\end{equation}
Let $N:=\lfloor M/\rho\rfloor+1$, $\tilde{I}_1:=\tilde{I}:=(0,\,\rho]$ and $\tilde{I}_j:=\tilde{I}_1+(j-1)\rho$, $j=2,\ldots,N$. Then
$\{\tilde{I}_j\}_{j=1}^N$ are mutually non-overlapping intervals and
$$(0, M)\subset\bigcup_{j=1}^N\tilde{I}_j.$$
Now define the mapping $\Phi$ by setting for any $f\in \mathcal F$ and $x\in\rrp$,
\begin{equation*}
\Phi(f)(x):=f_{\tilde{I}_1,\,\lz}\chi_{\tilde I_1}(x)+\sum_{j=2}^N\dfrac{1}{|\tilde{I}_j|}\dint_{\tilde{I}_j}f(z)\,dz\chi_{\tilde I_j}(x).
\end{equation*}
We first see that for $f\in\cf$, $\Phi(f)$ is well defined. In fact, if $x\in \tilde{I}_1$, then it follows from the H\"older inequality that
\begin{equation}\label{Phi-I1}
|\Phi(f)(x)|=\frac1{m_\lz(\tilde{I}_1)}\lf|\int_{\tilde{I}_1}f(y)\ytz\r|\le \frac1{[m_\lz(\tilde{I}_1)]^{1/p}}\|f\|_\lpz<\fz;
\end{equation}
while if $x\in \tilde{I}_j,$ $j=2,\ldots,N$, by another application of the H\"older inequality, we also have
\begin{equation}\label{Phi-Ij}
|\Phi(f)(x)|=\dfrac{1}{|\tilde{I}_j|}\lf|\dint_{\tilde{I}_j}f(z)\,dz\r|
\le\dfrac1{|\tilde{I}_j|}\|f\|_\lpz\lf[\int_{\tilde{I}_j}z^{-\frac{2\lz}pp'}\,dz\r]^{1/p'}<\fz.
\end{equation}

Let $\cb_N$ be the linear space spanned by $\{\chi_{\tilde I_j}\}_{j=1}^N$.
Then $\cb_N$ is a finite dimensional Banach space endowed with the norm $\|\cdot\|_\lpz$ for $p\in(1, \fz)$.
Observe that $\Phi(\cf)$ is a subset of $\cb_N$. Moreover, by \eqref{Phi-I1} and \eqref{Phi-Ij}, for any $f\in\mathcal F$,
%In fact, we only need to show that $\cb_N$ is complete in the norm $\|\cdot\|_\lpz$.
%To this end, assume that $\{f_k\}_k$ is a Cauchy sequence of $\cb_N$. Then by the completeness of $\lpz$,
%there exists a function $f\in\lpz$ and a subsequence of $\{f_k\}_k$ (we still write as $\{f_k\}_k$), such that for each $k$,
%$$\lf\|f_k-f\r\|_\lpz\le 2^{-k-2}\|f\|_\lpz.$$
%This implies that
%$$f=\sum_{k=1}^\fz(f_k-f_{k-1})\,\,{\rm in}\,\, \lpz,$$
%where $f_0:=0$. Because $f_k-f_{k-1}\in \mathcal B_N$,
%we see that
%$$f_k-f_{k-1}=\sum_{j=1}^N\big(c^k_j-c^{k-1}_j\big)\chi_{\tilde I_j},$$
%and
%\begin{eqnarray*}
%\lf|c^k_j-c^{k-1}_j\r|\lf[m_\lz\lf(\tilde I_j\r)\r]^{1/p}&\le&\lf\|f_k-f_{k-1}\r\|_\lpz\\
%&\le& \|f_k-f\|_\lpz+\|f_{k-1}-f\|_\lpz\\
%&\le& 2^{-k}\|f\|_\lpz.
%\end{eqnarray*}
%This implies that for each $j$,
%\begin{eqnarray*}
%\sum_{k=1}^\fz\big|c^k_j-c^{k-1}_j\big|&\le &\sum_{k=1}^\fz2^{-k}\lf[m_\lz\lf(\tilde I_j\r)\r]^{-1/p}\|f\|_\lpz
%\le\lf[m_\lz\lf(\tilde I_j\r)\r]^{-1/p}\|f\|_\lpz.
%\end{eqnarray*}
%Thus, we conclude that
%$$f=\sum_{k=1}^\fz(f_k-f_{k-1})=\sum_{k=1}^\fz\sum_{j=1}^N\big(c^k_j-c^{k-1}_j\big)\chi_{\tilde I_j}
%=\sum_{j=1}^N\sum_{k=1}^\fz\big(c^k_j-c^{k-1}_j\big)\chi_{\tilde I_j}=:\sum_{j=1}^N\wz c_j\chi_{\tilde I_j}.$$
%This shows the completeness of $\cb_N$.

%Since $\mathcal B_2$ is a finite dimensional Banach space, we
%have that for fixed $p\in(1, \fz)$ and any $f\in\mathcal B_2$,
%$$\|f\|_\ltz\sim\|f\|_\lpz,$$
%with the implicit constants independent of $f$.
%$$\Phi(f)=\sum_{j=2}f_{j,\,\lz}\chi_{\tilde I_j}+ \dfrac{1}{|\tilde{I}_j|}\dint_{\tilde{I}_j}f(z)\,dz,$$
%and
\begin{eqnarray*}
\|\Phi(f)\|_\lpz&\le&\lf|f_{\tilde I_1,\,\lz}\r|\lf[m_\lz\big(\tilde I_1\big)\r]^{1/p}
+\sum_{j=2}^N\lf| \dfrac{1}{|\tilde{I}_j|}\dint_{\tilde{I}_j}f(z)\,dz\r|\lf[m_\lz\big(\tilde I_j\big)\r]^{1/p}\\
&\ls&\|f\|_\lpz<\fz,
\end{eqnarray*}
where the implicit constant depends only on $\rho$, $p$, $N$ and $\lz$.
Thus, $\Phi(\mathcal F)$ is a bounded set of $(\mathcal B_N, \|\cdot\|_\lpz)$, and hence is totally bounded.

We now prove $\mathcal F$ is  totally bounded.
In fact, from the definition of $\Phi(f)$ and \eqref{f-vani-bd}, we find
\begin{eqnarray*}
\|f-\Phi(f)\|_{L^p(\mathbb R_+,\,dm_\lz)}&&<\epsilon/12+\|f\chi_{(0,\,M)}-\Phi(f)\|_{L^p(\mathbb R_+,\,dm_\lz)}\\
&&\le\epsilon/12+\lf(\sum_{j=1}^N\int_{\tilde{I}_j}|f(x)-\Phi(f)(x)|^p\xtz\r)^{1/p}.\\
\end{eqnarray*}
By the H\"older inequality, a change of variable and \eqref{f-equaconti-bd}, we see that
\begin{eqnarray*}
&&\sum_{j=2}^N\int_{\tilde{I}_j}|f(x)-\Phi(f)(x)|^p\xtz\\
&&\quad\le\sum_{j=2}^N\frac{1}{|\tilde{I}_j|}\int_{\tilde{I}_j}\int_{\tilde{I}_j}|f(x)-f(z)|^pdz\,\xtz\\
&&\quad=\sum_{j=2}^N\frac{1}{|\tilde{I}|}\int_{\tilde{I}_j}\lf[\int_{\{z\in \tilde{I}_j:\,z<x\}}+\int_{\{z\in \tilde{I}_j:\,z\ge x\}}\r]|f(x)-f(z)|^pdz\,\xtz\\
&&\quad\le\sum_{j=2}^N\frac{1}{|\tilde{I}|}\int_{\tilde{I}_j}\int_{\tilde{I}_1}\lf[|f(x)-f(x-y)|^p+|f(x)-f(x+y)|^p\r]\,dy\,\xtz\\
&&\quad=\frac{1}{|\tilde{I}|}\int_{\tilde{I}}\sum_{j=2}^N\int_{\tilde{I}_j}\lf[|f(x)-f(x-y)|^p+|f(x)-f(x+y)|^p\r]\xtz\,dy\\
&&\quad\le\frac{2}{|\tilde{I}|}\int_{\tilde{I}}\int_0^\infty|f(x)-f(x+y)|^p\xtz\,dy\\
&&\quad<\lf(\epsilon/12\r)^p.
\end{eqnarray*}
On the other hand, by the H\"older inequality, we see that
\begin{eqnarray*}
\int_{\tilde{I}_1}|f(x)-\Phi(f)(x)|^p\xtz
&=&\int_{\tilde{I}_1}\lf|f(x)-f_{\tilde I_{1,\,\lz}}\r|^p\xtz\\
&\le&\frac1{m_\lz(\tilde{I}_1)}\int_{\tilde{I}_1}\int_{\tilde{I}_1}|f(x)-f(z)|^pz^{2\lz}\,dz x^{2\lz}dx\\
&=&\frac{2}{m_\lz(\tilde{I})}\int_{\tilde{I}}\int_{x,\,z\in\tilde{I},\,z>x}|f(x)-f(z)|\ztz\,\xtz\\
&\le&\frac{2}{m_\lz(\tilde{I})}\int_0^\rho\int_0^{\rho-x}|f(x)-f(x+y)|^p\,(x+y)^{2\lz}dy\,\xtz\\
&\le&\frac{2\rho^{2\lz}}{m_\lz(\tilde{I})}\int_0^\rho\int_0^{\rho}|f(x)-f(x+y)|^p\,\xtz\,dy\\
&\le&2(2\lz+1)\|f(\cdot)-f(\cdot+y)\|_{L^p(\mathbb R_+,\,dm_\lz)}^p\\
&<&\lf(\epsilon/12\r)^p,
\end{eqnarray*}
where the last inequality follows from the estimate in \eqref{f-equaconti-bd}.

Combining these two inequalities above, we conclude that for any $f\in\cf$,
\begin{equation}\label{phi appro lpz}
\|f-\Phi(f)\|_{L^p(\mathbb R_+,\,dm_\lz)}<\epsilon/4.
\end{equation}
%and
%$$\|\Phi(f)\|_{L^p(\mathbb R_+,\,dm_\lz)}\le \epsilon/6+\|f\|_{L^p(\mathbb R_+,\,dm_\lz)}\le \epsilon+\sup_{f\in \mathcal{F}}\|f\|_{L^p(\mathbb R_+,\,dm_\lz)}<\infty.$$
%Thus, $\Phi(\cf)$ is a bounded subset of a finite dimensional Banach space, and hence $\Phi(\cf)$ is totally bounded.
By \eqref{phi appro lpz} and the linearity of $\Phi$, we further deduce that for any
$f,\,g\in\cf$ satisfying \eqref{phi lpz regularity},
$$\|f-g\|_{L^p(\mathbb R_+,\,dm_\lz)}\le \|f-\Phi(f)\|_\lpz+\|\Phi(f)-\Phi(g)\|_{L^p(\mathbb R_+,\,dm_\lz)}+\|\Phi(g)-g\|_\lpz<\epsilon.$$
Thus \eqref{totally bdd} holds and $\cf$ is totally bounded by Lemma \ref{l-total bdd}.

For the converse, assume that $\cf$ is totally bounded. For every $\epsilon>0$,
the existence of a finite $\epsilon-$cover of $\cf$ implies the boundedness of $\cf$, thus the condition ${\rm (a)}$ holds.

To show ${\rm (b)}$ holds, given $\epsilon>0$, let $\{U_1, \ldots, U_m\}$ be an $\epsilon-$cover of $\cf$,
and choose $g_j\in U_j$ for $j=1, \ldots, m$. Let $M>0$ such that
$$\int_M^\infty|g_j(x)|^p\xtz<\epsilon^p, \ \ \ \ j=1,\ldots, m.$$
If $f\in U_j$, then $\|f-g_j\|_{L^p(\mathbb R_+,\,dm_\lz)}<\epsilon$; and so
\begin{eqnarray*}
&&\lf(\int_M^\infty|f(x)|^p\xtz\r)^{1/p}\\
&&\quad\le\lf(\int_M^\infty|f(x)-g_j(x)|^p\xtz\r)^{1/p}+\lf(\int_M^\infty|g_j(x)|^p\xtz\r)^{1/p}\\
&&\quad<2\epsilon.
\end{eqnarray*}
Thus ${\rm (b)}$ holds.

For condition ${\rm (c)}$, given $\epsilon>0$, we pick an $\epsilon$-cover $\{U_1, \ldots, U_m\}$ of $\cf$.
Since $\cd$ is dense in $L^p(\mathbb R_+, dm_\lz)$, there exists $g_j\in U_j\cap \mathcal{D}$, for each $j=1, \ldots, m$. It is not hard to see that,
for every $g\in\mathcal D$,
$$\lim_{y\rightarrow 0^+}\int_0^\infty|g(x+y)-g(x)|^p\,x^{2\lz}dx=0.$$
Then there exists $\rho>0$ such that
$$\int_0^\infty|g_j(x+y)-g_j(x)|^p\xtz<\epsilon^p,\ \ \ \ y\in(0,\rho),\ j=1,\ldots, m.$$
Moreover, for any $f\in\cf$, we see that $f\in U_j$ for certain $j=1,\ldots, m$ and hence,
\begin{eqnarray*}
&&\lf(\int_0^\infty|f(x+y)-f(x)|^p\xtz\r)^{1/p}\\
&&\quad\le\lf(\int_0^\infty|f(x+y)-g_j(x+y)|^p\xtz\r)^{1/p}+\lf(\int_0^\infty|g_j(x+y)-g_j(x)|^p\xtz\r)^{1/p}\\
&&\quad\quad+\lf(\int_0^\infty|g_j(x)-f(x)|^p\xtz\r)^{1/p}\\
&&\quad\le\lf(\int_0^\infty|f(x+y)-g_j(x+y)|^p\,(x+y)^{2\lz}dx\r)^{1/p}+\lf(\int_0^\infty|g_j(x+y)-g_j(x)|^p\xtz\r)^{1/p}\\
&&\quad\quad+\lf(\int_0^\infty|g_j(x)-f(x)|^p\xtz\r)^{1/p}\\
&&\quad\le\lf(\int_0^\infty|g_j(x+y)-g_j(x)|^p\xtz\r)^{1/p}+2\lf(\int_0^\infty|g_j(x)-f(x)|^p\xtz\r)^{1/p}\\
&&\quad<5\epsilon.
\end{eqnarray*}
This finishes the proof of Theorem \ref{t-fre kol}.
\end{proof}

\section{The proof of Theorem \ref{t-riesz compact}}\label{s5}

In this section, we give the proof of Theorem \ref{t-riesz compact}. To begin
with, we first recall the following boundedness of $[b,\riz]$ established in \cite{dlwy}.

\begin{lem}{\rm(\cite{dlwy})}\label{l-bdd of riz}
Let $b\in \cup_{q>1}L^q_{\rm loc}(\mathbb{R}_+,\,dm_\lz)$ and $p\in(1, \fz)$.
 Then $b\in\bmoz$ if and only if $[b,\riz]$ is bounded on $\lpz$.
Moreover, there exists a positive constant $C\in(1, \fz)$
such that
\begin{eqnarray*}
C^{-1}\|b\|_{\bmoz}&\le&\left\|[b, \riz]\right\|_{\lpz\to\lpz} \le C\|b\|_\bmoz.
\end{eqnarray*}
\end{lem}

Before giving the proof of Theorem \ref{t-riesz compact},  we first obtain a lemma for
the upper and lower bounds of integrals of $[b, \riz]f_j$ on certain intervals.
%, where $\{f_j\}_j$ is a bounded subset of $\lpz$ and $b\in\bmoz$.
To this end,
we recall the median value in \cite{st,hyy}, see also \cite{j65,s79,j,jt}.
For $f\in L^1_{\rm loc}(\mathbb R_+,\,dm_\lz)$ and $I\subset \rrp$, let $\alpha_I(f)$ be a real number such that
$$\inf_c\frac{1}{m_\lz(I)}\int_I|f(x)-c|dm_\lz(x)$$
is attained. Note that $\frac{1}{m_\lz(I)}\int_I|f-c|dm_\lz$ is uniformly continuous in $c$, so such $\alpha_I(f)$ exists and may not be unique.
Moreover, as in \cite[p. 30]{j} where the setting of $(\mathbb R_+,\,|\cdot|,\,dx)$ was considered, $\alpha_I(f)$ satisfies that
\begin{equation}\label{median value-1}
m_\lz(\{x\in I: f(x)>\alpha_I(f)\})\le m_\lz(I)/2
\end{equation}
and
\begin{equation}\label{median value-2}
m_\lz(\{x\in I: f(x)<\alpha_I(f)\})\le m_\lz(I)/2.
\end{equation}
In fact, if $\alpha_I(f)$ does not satisfy \eqref{median value-1}, then
$$m_\lz(\{x\in I: f(x)>\alpha_I(f)\})> m_\lz(I)/2.$$
%Thanks to
%$$m_\lz(\{x\in I: f(x)>\alpha_I(f)\})=\lim_{k\rightarrow\infty}m_\lz\lf(\lf\{x\in I: f(x)>\alpha_I(f)+\frac{1}{k}\r\}\r).$$
Take $\varepsilon>0$ small enough such that
$$m_\lz(\{x\in I: f(x)>\alpha_I(f)+\varepsilon\})> m_\lz(I)/2.$$
We define $I_1:=\{x\in I: f(x)>\alpha_I(f)+\varepsilon\}$ and $I_2:=I\setminus I_1$.
%and set
%$${\rm I}:=\int_I|f(x)-\alpha_I(f)|dm_\lz(x),$$
%and
%$${\rm II}:=\int_I|f(x)-(\alpha_I(f)+\varepsilon)|dm_\lz(x).$$
Then
\begin{eqnarray*}
&&\int_I|f(x)-\alpha_I(f)|dm_\lz(x)-\int_I|f(x)-(\alpha_I(f)+\varepsilon)|dm_\lz(x)\\
&&\quad=\int_{I_1}|f(x)-\alpha_I(f)|\,dm_\lz(x)+\int_{I_2}|f(x)-\alpha_I(f)|\,dm_\lz(x)\\
&&\qquad-\int_{I_1}(f(x)-\alpha_I(f)-\varepsilon)\,dm_\lz(x)-\int_{I_2}(\alpha_I(f)+\varepsilon-f(x))\,dm_\lz(x)\\
&&\quad\geq\varepsilon\lf(m_\lz(I_1)-m_\lz(I_2)\r)\\
&&\quad>0.
\end{eqnarray*}
This violates the choice of $\alpha_I(f)$. The proof of \eqref{median value-2} is similar and omitted.

Moreover, by the choice of $\alpha_I(f)$ and Definition \ref{d-bmo}, it is easy to see that for any interval $I\subset \rrp$,
\begin{equation}\label{equi osci}
\mofi\sim\frac1{m_\lz(I)}\int_I\lf|f(y)-\az_I(f)\r|\,\ytz.
\end{equation}

\begin{lem}\label{l-cmo-contra}
Assume that $b\in\bmoz$ with $\|b\|_\bmoz=1$ and there exist $\dz\in(0, \fz)$
and a sequence $\{I_j\}_{j=1}^\fz:=\{I(x_j, r_j)\}_j$ of intervals such that for each $j$,
\begin{equation}\label{lower bdd osci}
M_\lz(b, I_j)>\dz.
\end{equation}
Then there exist functions $\{f_j\}_j\subset \lpz$,  positive constants $A_1>4$, $\wz C_0$, $\wz C_1$ and $\wz C_2$
such that for any integers $j$ and $k\ge\lfloor\log_2 A_1\rfloor$, $\|f_j\|_\lpz \le \wz C_0$,
\begin{equation}\label{lower upper lpbdd riesz comm}
\int_{I_j^k}\lf|\lf[b, \riz\r]f_j(y)\r|^p\ytz\geq\wz C_1\dz^p\frac{[m_\lz(I_j)]^{p-1}}{[m_\lz(2^kI_j)]^{p-1}},
\end{equation}
%\begin{equation}\label{lower upper lpbdd riesz comm}
%\wz C_1\dz^p\frac{[m_\lz(I_j)]^{p-1}}{[m_\lz(2^kI_j)]^{p-1}}\le
%\int_{I_j^k}\lf|\lf[b, \riz\r]f_j(y)\r|^p\ytz\le \wz C_2 \frac{[m_\lz(I_j)]^{p-1}}{[m_\lz(2^kI_j)]^{p-1}},
%\end{equation}
where $I_j^k:=\lf(x_j+2^kr_j,\,x_j+2^{k+1}r_j\r);$
and
\begin{equation}\label{lower upper lpbdd riesz comm2}
\int_{2^{k+1}I_j\setminus 2^k I_j}\lf|\lf[b, \riz\r]f_j(y)\r|^p\ytz\le \wz C_2 \frac{[m_\lz(I_j)]^{p-1}}{[m_\lz(2^kI_j)]^{p-1}}.
\end{equation}
%Here we assume that  $\|b\|_\bmoz=1$.
%  \item [iii)]For any $\varepsilon\in(0, \fz)$, there exists $\beta\in(0, 1)$ such that
%\begin{equation}\label{upper lpbdd riesz comm2}
%\int_{E_j}\lf|\lf[b, \riz\r]f_j(y)\r|^p\ytz\le \varepsilon
%\end{equation}
%holds for every interval $E_j$ satisfying
%\begin{equation}\label{e set-2}
%E_j\subset(x_j+A_1r_j, x_j+A_2r_j) \,{\rm and} \,E_j\cap I_j^k\not=\emptyset
%\end{equation}
% for some $k$, and
%\begin{equation}\label{e set}
%\ m_\lz(E_j)/m_\lz(I_j)<C_2\beta^{2\lambda+1},
%\end{equation}
%where $C_2$ depending only on $\lz$.

\end{lem}

\begin{proof}
For each $j$, define the function $f_j$ as follows:
$$f^1_j:=\chi_{I_{j,\,1}}-\chi_{I_{j,\,2}}:=\chi_{\{x\in I_j:\, b(x)>\az_{I_j}(b)\}}-\chi_{\{x\in I_j:\, b(x)<\az_{I_j}(b)\}},$$
%otherwise
%$$f^1_j:=\chi_{\{x\in I_j:\, b(x)<\az_{I_j}(b)\}}-\chi_{\{x\in I_j:\, b(x)>\az_{I_j}(b)\}}.$$
$f^2_j:=a_j\chi_{I_j}$ and
$$f_j:=\lf[m_\lz\lf(I_j\r)\r]^{-1/p}\lf(f^1_j-f^2_j\r),$$
 where $a_j$ is a constant such that
 \begin{equation}\label{fj proper-2}
 \inzf f_j(x)\xtz=0.
  \end{equation}
 Then by the definition of $a_j$, \eqref{median value-1} and \eqref{median value-2}, we see that $|a_j|\le 1/2$. Moreover,
we also have that $\supp(f_j)\subset \bar{I}_j$, and that for any $y\in I_j$,
\begin{equation}\label{fj proper-1}
 f_j(y)\lf[b(y)-\az_{I_j}(b)\r]\ge0.
 \end{equation}
On the other hand, since $|a_j|\le1/2$, we see that for any $y\in (I_{j,\,1}\cup I_{j,\,2})$,
  \begin{equation}\label{fj proper-3}
|f_j(y)|\sim \lf[m_\lz\lf(I_j\r)\r]^{-1/p}.
  \end{equation}
Moreover, since $\supp(f_j)\subset \bar{I}_j$, we have that $\|f_j\|_\lpz\ls 1$.
Observe that
  \begin{equation}\label{com equiv}
[b, \riz]f=\riz\lf([b-\az_{I_j}(b)]f\r)-\lf[b-\az_{I_j}(b)\r]\riz(f).
\end{equation}
Let  $A_1>4$ large enough.
%and
%$$\wz I_j:=(x_j+A_1r_j, x_j+A_2r_j).$$
Then for any integer $k\ge\lfloor\log_2 A_1\rfloor$,
\begin{eqnarray}\label{ijk set inclu}
2^{k+1}I_j\subset8I_j^k&=&\lf(x_j-\frac{5}{2}\cdot2^kr_j,\,x_j+\frac{11}{2}\cdot2^kr_j\r)\cap\mathbb R_+\subset2^{k+3}I_j,
\end{eqnarray}
and by \eqref{reverse doubl},
\begin{equation}\label{ijk meas}
m_\lz\lf(I_j^k\r)\sim m_\lz\lf(2^{k}I_j\r).
\end{equation}

We first prove the inequality \eqref{lower upper lpbdd riesz comm}. By the fact that $|y-x_j|>2|z-x_j|$ for any $y\in\mathbb R_+\setminus 2I_j$ and $z\in I_j$,
\eqref{fj proper-3}, \eqref{fj proper-2}, \eqref{cz kernel condition-2}, we see that,
 \begin{eqnarray}\label{upper bdd riesz ope}
\lf|\lf[b(y)-\az_{I_j}(b)\r]\riz(f_j)(y)\r|&=&\lf|b(y)-\az_{I_j}(b)\r|\lf|\int_{I_j}\lf[\riz(y,z)-\riz(y,x_j)\r]f_j(z)\ztz\r|\noz\\
&\le&\lf|b(y)-\az_{I_j}(b)\r|\int_{I_j}|\riz(y,z)-\riz(y,x_j)||f_j(z)|\ztz\noz\\
&\ls& r_j\lf[m_\lz\lf(I_j\r)\r]^{1/p'}\frac{|b(y)-\az_{I_j}(b)|}{|x_j-y|m_\lz(I(y,|y-x_j|))}.
\end{eqnarray}
%In particular, when $y-x_j>2|z-x_j|$ for any $y\in I_j^k$ and $z\in I_j$, \eqref{upper bdd riesz ope} also holds.

Moreover, by the well known John-Nirenberg inequality (\cite[p. 594]{cw77}) and \eqref{reverse doubl},
we conclude that  for each $k\in \nn$ and $I\subset\rrp$,
\begin{eqnarray}\label{b-bmo-bdd}
&&\int_{2^{k+1}I}\lf|b(y)-\az_{I}(b)\r|^p\,\ytz\noz\\
&&\quad\ls \int_{2^{k+1}I}\lf|b(y)-\az_{2^{k+1}I}(b)\r|^p\,\ytz
+m_\lz\lf(2^{k+1}I\r)\lf|\az_{2^{k+1}I}(b)-\az_{I}(b)\r|^p\noz\\
&&\quad\ls k^pm_\lz\lf(2^kI\r).
\end{eqnarray}
By this fact, the fact that for any $x$ and $y$,
\begin{equation}\label{equiv meas}
m_\lz(I(x, |x-y|))\sim m_\lz(I(y, |x-y|)),
\end{equation}
 \eqref{upper bdd riesz ope}, \eqref{ijk set inclu} and \eqref{reverse doubl}, we see that
there exists a positive constant $C_4$, such that for any $k\in \nn$,
\begin{eqnarray}\label{upper bdd com}
\int_{I_j^k}\lf|\lf[b(y)-\az_{I_j}(b)\r]\riz(f_j)(y)\r|^p\ytz
&&\ls\int_{I_j^k}\frac{[m_\lz\lf(I_j\r)]^{p-1}\lf|b(y)-\az_{I_j}(b)\r|^p}{2^{kp}[m_\lz(I(x_j,\,y-x_j))]^p}\ytz\noz\\
&&\ls \frac1{2^{kp}}\frac{[m_\lz\lf(I_j\r)]^{p-1}}{[m_\lz(2^{k}I_j)]^p}\int_{2^{k+3}I_j}\lf|b(y)-\az_{I_j}(b)\r|^p\ytz\noz\\
&&\ls\frac{k^p}{2^{kp}}\frac{[m_\lz(I_j)]^{p-1}}{[m_\lz(2^{k}I_j)]^p}m_\lz\lf(2^{k+3}I_j\r)\noz\\
&&\le C_4\frac{k^p}{2^{kp}}\frac{[m_\lz(I_j)]^{p-1}}{[m_\lz(2^{k}I_j)]^{p-1}}.
\end{eqnarray}

Next, observe that $y>z$ for any $y\in I_j^k$ and $z\in I_j$.
By Proposition \ref{p-lower bdd riesz}, \eqref{fj proper-1}, \eqref{fj proper-3},
\eqref{equi osci} and  \eqref{lower bdd osci}, we have that
 \begin{eqnarray*}
\lf|\riz\lf[(b-\az_{I_j}(b))f_j\r](y)\r|&=&\int_{(I_{j,\,1}\cup I_{j,\,2})}|\riz(y,z)|\lf|\lf[b(z)-\az_{I_j}(b)\r]f_j(z)\r|\ztz\\
&\gs& \lf[m_\lz\lf(I_j\r)\r]^{-1/p}\int_{I_j}\frac{|b(z)-\az_{I_j}(b)|}{m_\lz(I(y,|y-z|))}z^{2\lz}\,dz\\
&\gs& \dz\lf[m_\lz\lf(I_j\r)\r]^{1/p'}\frac1{m_\lz(I(y,|y-x_j|))}.
\end{eqnarray*}
From this,  \eqref{equiv meas} and \eqref{ijk meas}, we deduce that there exists a positive constant $C_5$ such that
\begin{eqnarray}\label{low bdd com}
\int_{I_j^k}\lf|\riz\lf[(b-\az_{I_j}(b))f_j\r](y)\r|^p\ytz
&&\gs \dz^p\lf[m_\lz\lf(I_j\r)\r]^{p-1}\int_{I_j^k}\frac1{[m_\lz(I(y,|y-x_j|))]^p}\ytz\noz\\
&&\gs\dz^p\frac{[m_\lz\lf(I_j\r)]^{p-1}}{[m_\lz(2^{k}I_j)]^p}m_\lz\lf(I_j^k\r)\noz\\
&&\ge \dz^pC_5\frac{[m_\lz\lf(I_j\r)]^{p-1}}{[m_\lz(2^{k}I_j)]^{p-1}}.
\end{eqnarray}
Take $A_1$ large enough such that for any integer $k\ge \lfloor\log_2 A_1\rfloor$,
$$C_5\frac{\dz^p}{2^{p-1}}-C_4\frac{k^p}{2^{kp}}\ge C_5\frac{\dz^p}{2^p}.$$
By \eqref{com equiv}, \eqref{low bdd com} and \eqref{upper bdd com},  we conclude that for any integer $k\ge\lfloor\log_2 A_1\rfloor$,
\begin{eqnarray*}
&&\int_{I_j^k}\lf|[b, \riz]f_j(y)\r|^p\,\ytz\noz\\
&&\quad\ge\bigg[\frac1{2^{p-1}}\int_{I_j^k}\lf|\riz\lf[(b-\az_{I_j}(b))f_j\r](y)\r|^p\ytz\noz\\
&&\quad\quad-\int_{I_j^k}\lf|\lf[b(y)-\az_{I_j}(b)\r]\riz(f_j)(y)\r|^p\ytz\bigg]\noz\\
&&\quad\ge\lf(C_5\frac{\dz^p}{2^{p-1}}-C_4\frac{k^p}{2^{kp}}\r)
\frac{[m_\lz\lf(I_j\r)]^{p-1}}{[m_\lz(2^{k}I_j)]^{p-1}}\ge C_5\frac{\dz^p}{2^p}
\frac{[m_\lz\lf(I_j\r)]^{p-1}}{[m_\lz(2^{k}I_j)]^{p-1}}.
\end{eqnarray*}
This shows the  inequality \eqref{lower upper lpbdd riesz comm}.

Now we show the inequality \eqref{lower upper lpbdd riesz comm2}.
From $\supp(f_j)\subset \bar{I}_j$, \eqref{cz kernel condition-1}, \eqref{equi osci} and \eqref{fj proper-3},
we deduce that for any $y\in \rrp\setminus 2I_j$,
\begin{eqnarray*}
\lf|\riz\lf[(b-\az_{I_j}(b))f_j\r](y)\r|&\ls& \lf[m_\lz\lf(I_j\r)\r]^{-1/p}\int_{I_j}\frac{|b(z)-\az_{I_j}(b)|}{m_\lz(I(y,|y-z|))}z^{2\lz}\,dz\noz\\
&\ls& \lf[m_\lz\lf(I_j\r)\r]^{1/p'}\frac1{m_\lz(I(y,|y-x_j|))},
\end{eqnarray*}
from which together with \eqref{upper bdd riesz ope}, \eqref{b-bmo-bdd} and \eqref{reverse doubl},
it follows that for any $k\ge\lfloor \log_2A_1\rfloor$,
\begin{eqnarray*}
&&\int_{2^{k+1}I_j\setminus 2^kI_j}\lf|[b, \riz]f_j(y)\r|^p\,\ytz\\
&&\quad\ls\int_{2^{k+1}I_j\setminus 2^kI_j}\lf|\riz\lf([b-\az_{I_j}(b)]f_j\r)(y)\r|^p\ytz\\
&&\quad\quad+\int_{2^{k+1}I_j\setminus 2^kI_j}\lf|\lf[b-\az_{I_j}(b)\r]\riz(f_j)(y)\r|^p\ytz\\
&&\quad\ls [m_\lz(I_j)]^{p-1}\int_{2^{k+1}I_j\setminus 2^kI_j}\frac{1}{[m_\lz(I(y, |y-x_j|))]^p}\ytz \\
&&\quad\quad+r_j^p [m_\lz(I_j)]^{p-1}\int_{2^{k+1}I_j\setminus 2^kI_j}\frac{|b(y)-\az_{I_j}(b)|^p}{|y-x_j|^p [m_\lz(I(y, |y-x_j|))]^p}\ytz\\
&&\quad\ls\frac{k^p[m_\lz(I_j)]^{p-1}m_\lz(2^{k+1}I_j)}{2^{kp}[m_\lz(2^{k}I_j)]^p}+\frac{[m_\lz(I_j)]^{p-1}m_\lz(2^{k+1}I_j)}{[m_\lz(2^{k}I_j)]^p}\\
&&\quad\leq\tilde{C}_2\frac{[m_\lz(I_j)]^{p-1}}{[m_\lz(2^{k}I_j)]^{p-1}}.
\end{eqnarray*}
We finish the proof of Lemma \ref{l-cmo-contra}.
\end{proof}

\begin{proof}[\bf Proof of Theorem \ref{t-riesz compact}]

{\bf Sufficiency:}

We use the idea in \cite{u78}.
We first show that if $[b, \riz]$ is a compact operator on $\lpz$, then $b\in \cmoz$.
Since $[b, \riz]$ is compact on $\lpz$, $[b, \riz]$ is bounded on $\lpz$.
By Lemma \ref{l-bdd of riz}, we see that $b\in\bmoz$.
Without loss of generality, we may assume that $\|b\|_\bmoz=1$.
To show $b\in\cmoz$, we use a contradiction argument via Theorem \ref{t-cmo char}.
Observe that if $b\notin \cmoz$, $b$ does not satisfy at least one of (i)-(iii) in Theorem \ref{t-cmo char}.

We now consider the following three cases.

Case i), $b$ does not satisfy (i) in Theorem \ref{t-cmo char}. Then there exists $\dz\in(0, \fz)$
and a sequence $\{I_j\}_{j=1}^\fz$ of intervals satisfying \eqref{lower bdd osci}
and that $m_\lz\lf(I_j\r)\to0$ as $j\to\fz$.
Let $f_j$, $\wz C_1$, $\wz C_2$, $A_1$ be as in Lemma \ref{l-cmo-contra} and $A_2>A_1$ large enough
such that
$$A_3:=8^{(1-p)(2\lz+1)}\wz C_1\dz^pA_1^{(1-p)(2\lz+1)}>
\frac{2\wz C_2}{1-[\min(2^{2\lz},\,2)]^{1-p}}\frac1{[\min(2^{2\lz},\,2)]^{\lfloor\log_2 A_2\rfloor(p-1)}}.$$
Since $m_\lz\lf(I_j\r)\to 0$ as $j\to\fz$, we may choose a subsequence $\{I_{j_\ell}^{(1)}\}$ of $\{I_j\}$ such that
\begin{equation}\label{descreasing interval}
\frac{m_\lz(I_{j_{\ell+1}}^{(1)})}{m_\lz(I_{j_{\ell}}^{(1)})}<\frac1{(2A_2)^{2\lambda+1}}.
\end{equation}
For fixed $\ell$, $m\in \mathbb N$, denote
$$\mathcal J:=\lf(x_{j_\ell}^{(1)}+A_1r_{j_\ell}^{(1)}, x_{j_\ell}^{(1)}+A_2r_{j_\ell}^{(1)}\r),$$
$$\mathcal J_1:=\mathcal J\setminus \lf\{y\in\mathbb R_+: \lf|y-x_{j_{\ell+m}}^{(1)}\r|\le A_2r_{j_{\ell+m}}^{(1)}\r\}$$
and
$$\mathcal J_2:=\lf\{y\in\mathbb R_+: \lf|y-x_{j_{\ell+m}}^{(1)}\r|>A_2r_{j_{\ell+m}}^{(1)}\r\}.$$
Note that
$$\mathcal J_1\subset\lf\{y\in\mathbb R_+: \lf|y-x_{j_{\ell}}^{(1)}\r|\le A_2r_{j_{\ell}}^{(1)}\r\}\cap\mathcal J_2
\,\,{\rm and}\,\, \mathcal J_1=\mathcal J\setminus(\mathcal J\setminus\mathcal J_2).$$
 We then have
\begin{eqnarray}\label{low lpbdd comparing com}
&&\|\lf[b,\riz\r](f_{j_\ell})-\lf[b,\riz\r](f_{j_{\ell+m}})\|_{L^p(\mathbb R_+,\,dm_\lz)}\noz\\
&&\quad\ge\lf(\int_{\mathcal J_1}\lf|\lf[b,\riz\r](f_{j_\ell})(y)-\lf[b,\riz\r](f_{j_{\ell+m}})(y)\r|^p\ytz\r)^{1/p}\noz\\
&&\quad\ge\lf(\int_{\mathcal J_1}\lf|\lf[b,\riz\r](f_{j_\ell})(y)\r|^p\ytz\r)^{1/p}
-\lf(\int_{\mathcal J_2}\lf|\lf[b,\riz\r](f_{j_{\ell+m}})(y)\r|^p\ytz\r)^{1/p}\noz\\
&&\quad=\lf(\int_{\mathcal J\setminus(\mathcal J\setminus\mathcal J_2)}\lf|\lf[b,\riz\r](f_{j_\ell})(y)\r|^p\ytz\r)^{1/p}
-\lf(\int_{\mathcal J_2}\lf|\lf[b,\riz\r](f_{j_{\ell+m}})(y)\r|^p\ytz\r)^{1/p}\noz\\
&&\quad=:{\rm F_1}-{\rm F_2}.
\end{eqnarray}

We first consider the term ${\rm F_1}$. To begin with, we now estimate the measure of $\mathcal J\setminus\mathcal J_2$.  Assume that $E_{j_\ell}:=\mathcal J\setminus\mathcal J_2\not=\emptyset$.
Then $E_{j_\ell}\subset A_2I^{(1)}_{j_{\ell+m}}$.
Hence, we have
\begin{eqnarray}\label{ee1}
m_\lz\lf(E_{j_\ell}\r)\le m_\lz\lf(A_2I^{(1)}_{j_{\ell+m}}\r)\le (2A_2)^{2\lz+1}m_\lz\lf(I^{(1)}_{j_{\ell+m}}\r)<m_\lz\lf(I^{(1)}_{j_{\ell}}\r),
\end{eqnarray}
where the second inequality follows from the doubling condition \eqref{reverse doubl}, and the last inequality follows from \eqref{descreasing interval}.

Now let
$$I_{j_\ell}^{k}:=\lf(x^{(1)}_{j_\ell}+2^kr^{(1)}_{j_\ell},\,x^{(1)}_{j_\ell}+2^{k+1}r^{(1)}_{j_\ell}\r),\quad {\rm with\ } k\geq 1.$$
Then by \eqref{ijk meas} and \eqref{reverse doubl},
$$
m_\lz\lf(I_{j_\ell}^{k}\r)\gs \lf[\min\lf(2^{2\lz}, 2\r)\r]^k m_\lz\lf(I^{(1)}_{j_\ell}\r),
$$
which, together with \eqref{ee1}, implies that
$$ m_\lz\lf(I_{j_\ell}^{k}\r) \gs m_\lz\lf(E_{j_\ell}\r). $$

From this fact, it follows that there exist at most two intervals, $I_{j_\ell}^{k_0}$ and $I_{j_\ell}^{k_0+1}$,
such that $E_{j_\ell}\subset (I_{j_\ell}^{k_0}\cup
I_{j_\ell}^{k_0+1})$. By \eqref{lower upper lpbdd riesz comm} and \eqref{reverse doubl},
\begin{eqnarray*}
{\rm F}_1^p&&\ge\sum_{k=\lfloor\log_2 A_1\rfloor+1,\,k\not=k_0,\,k_0+1}^{\lfloor\log_2 A_2\rfloor}
\int_{I_{j_\ell}^k}\lf|\lf[b,\riz\r](f_{j_\ell})(y)\r|^p\ytz\\
&&\ge\wz C_1\dz^p\sum_{k=\lfloor\log_2 A_1\rfloor+1,\,k\not=k_0,\,k_0+1}^{\lfloor\log_2 A_2\rfloor}
\frac{[m_\lz(I^{(1)}_{j_\ell})]^{p-1}}{[m_\lz(2^{k}I^{(1)}_{j_\ell})]^{p-1}}\noz\\
&&\ge\wz C_1\dz^p\sum_{k=\lfloor\log_2 A_1\rfloor+3}^{\lfloor\log_2 A_2\rfloor}\frac{1}{2^{k(p-1)(2\lz+1)}}\noz\\
&&\ge 8^{(1-p)(2\lz+1)}\wz C_1\dz^pA_1^{(1-p)(2\lz+1)}=A_3.
\end{eqnarray*}
If  $E_{j_\ell}:=\mathcal J\setminus\mathcal J_2=\emptyset$, the inequality above still holds.

%From \eqref{descreasing interval}, it yields
%\begin{eqnarray}\label{set I condition_1}
%\frac{m_\lz(\mathcal J\setminus\mathcal J_2)}{m_\lz(I_{j_\ell}^{(1)})}\ls\frac{A_3^{2\lz+1}m_\lz(I_{j_{\ell+m}}^{(1)})}{m_\lz(I_{j_\ell}^{(1)})}
%%&&\quad\ls A_3^{2\lz+1}\lf(\beta/A_3\r)^{(2\lz+1)m}\ls A_3^{2\lz+1}\lf(\beta/A_3\r)^{(2\lz+1)}\noz\\
%\ls\beta^{2\lz+1}.
%\end{eqnarray}
On the other hand, from \eqref{lower upper lpbdd riesz comm2} and \eqref{reverse doubl}, we deduce that
\begin{eqnarray*}
{\rm F}_2^p&&\le\sum_{k=\lfloor\log_2 A_2\rfloor}^{\fz}
\int_{2^{k+1}I_{j_{\ell+m}}^{(1)} \setminus 2^{k}I_{j_{\ell+m}}^{(1)}}\lf|\lf[b,\riz\r](f_{j_{\ell+m}})(y)\r|^p\ytz\\
&&\le\wz C_2\sum_{k=\lfloor\log_2 A_2\rfloor}^\fz
\frac{[m_\lz(I_{j_{\ell+m}}^{(1)})]^{p-1}}{[m_\lz(2^{k}I_{j_{\ell+m}}^{(1)})]^{p-1}}\noz\\
&&\le\wz C_2\sum_{k=\lfloor\log_2 A_2\rfloor}^\fz\frac{1}{[\min(2^{2\lz},\,2)]^{k(p-1)}}\\
&&\le \frac{\wz C_2}{1-[\min(2^{2\lz},\,2)]^{1-p}}\frac1{[\min(2^{2\lz},\,2)]^{\lfloor\log_2 A_2\rfloor(p-1)}}<A_3/2.
\end{eqnarray*}
%By \eqref{set I condition_1} and applying \eqref{upper lpbdd riesz comm2} for $E_j:=\mathcal J\setminus\mathcal J_2$, we get
%\begin{equation}\label{upper lpbdd riesz comm3}
%\int_{\mathcal J\setminus\mathcal J_2}\lf|\lf[b,\riz\r](f_{j_\ell})(y)\r|^p\ytz\le A_3/4.
%\end{equation}
By these two inequalities and \eqref{low lpbdd comparing com}, we get
\begin{eqnarray*}
&&\|\lf[b,\riz\r](f_{j_\ell})-\lf[b,\riz\r](f_{j_{\ell+m}})\|_{L^p(\mathbb R_+,\,dm_\lz)}\gs(A_3)^{1/p}.
\end{eqnarray*}
Thus, $\{[b,\riz]f_j\}_{j}$ is not relatively compact in $\lpz$, which implies that
$[b, \riz]$ is not compact on $\lpz$. Therefore, $b$ satisfies condition (i).

Case ii), $b$ violates (ii) in Theorem \ref{t-cmo char}. In this case, we also have that there exist
$\delta\in(0, \fz)$ and a sequence $\{I_j\}$ of intervals satisfying \eqref{lower bdd osci}
and that $m_\lz(I_j)\rightarrow\infty$ as $j\rightarrow\infty$. We take a subsequence $\{I_{j_\ell}^{(2)}\}$ of $\{I_j\}$ such that
\begin{equation}\label{increasing interval}
\frac{m_\lz\lf(I_{j_{\ell}}^{(2)}\r)}{m_\lz\lf(I_{j_{\ell+1}}^{(2)}\r)}<\frac1{(2A_2)^{2\lambda+1}}.
\end{equation}
We can use a similar method as in the previous case and redefine our sets in a reversed order. That is, for fixed $\ell$ and $m$, let
$$\widetilde{\mathcal J}:=\lf(x_{j_{\ell+m}}^{(2)}+A_1r_{j_{\ell+m}}^{(2)}, x_{j_{\ell+m}}^{(2)}+A_2r_{j_{\ell+m}}^{(2)}\r),$$
$$\widetilde{\mathcal J_1}:=\widetilde{\mathcal J}\setminus \lf\{y\in\mathbb R_+: \lf|y-x_{j_{\ell}}^{(2)}\r|\le A_2r_{j_{\ell}}^{(2)}\r\}$$
and
$$\widetilde{\mathcal J_2}:=\lf\{y\in\mathbb R_+: \lf|y-x_{j_{\ell}}^{(2)}\r|>A_2r_{j_{\ell}}^{(2)}\r\}.$$
Then we have that
$$\widetilde{\mathcal J_1}\subset\lf\{y\in\mathbb R_+: \lf|y-x_{j_{\ell+m}}^{(2)}\r|\le A_2r_{j_{\ell+m}}^{(2)}\r\}\cap\widetilde{\mathcal J_2}
\,\, {\rm and}\,\,
\widetilde{\mathcal J_1}=\widetilde{\mathcal J}\setminus\lf(\widetilde{\mathcal J}\setminus\widetilde{\mathcal J_2}\r).$$
%and
%$$\frac{m_\lz\lf(\widetilde{\mathcal J_2}^c\cap\widetilde{\mathcal J}\r)}{m_\lz\lf(I_{j_{\ell+m}}^{(2)}\r)}\ls\frac{A_2^{2\lz+1}m_\lz(I_{j_\ell}^{(2)})}{m_\lz(I_{j_{\ell+m}}^{(2)})}<\beta^{2\lz+1},$$
As in Case i), by Lemma \ref{l-cmo-contra} and \eqref{increasing interval},
we see that $[b, \riz]$ is not compact on $\lpz$. This contradiction implies that
$b$ satisfies (ii) of Theorem \ref{t-cmo char}.

Case iii),  condition (iii) in Theorem \ref{t-cmo char} does not hold for $b$. Then there exists $\delta>0$ such that for any $R>0$, there exists $I\subset[R,\,\infty)$ with $M_\lz(b,\,I)>\delta$.
We claim that for the $\delta$ above, there exists a sequence $\{I^{(3)}_j\}_j$ of intervals such that for any $j$,
\begin{equation}\label{lower bdd osci-2}
M_\lz(b,\,I^{(3)}_j)>\delta,
\end{equation}
and that for any $\ell\neq m$,
\begin{equation}\label{I-j-pro}
A_2I^{(3)}_{\ell}\bigcap A_2I^{(3)}_m=\emptyset.
\end{equation}
In fact, let $C_\delta>0$ to be determined later. Then for $R_1>C_\delta$,
there exists an interval $I^{(3)}_1:=I(x_1,\,r_1)\subset[R_1,\,\infty)$ such that \eqref{lower bdd osci-2} holds. Similarly, for $R_j:=x_{j-1}+4A_2C_\delta$, $j=2, 3, \ldots$, there exists $I^{(3)}_j:=I(x_j,\,r_j)\subset[R_j,\,\infty)$ satisfying \eqref{lower bdd osci-2}.
Repeating this procedure, we obtain $\{I^{(3)}_j\}_j$ satisfying \eqref{lower bdd osci-2} for each $j$. Moreover,
as $b$ satisfies the condition (ii) in Theorem \ref{t-cmo char}, for $\delta$ aforementioned, there exists a constant $\widetilde{C_\delta}$ such that $$M_\lz(b,\,I)<\delta$$
for any interval $I$ satisfying $m_\lz(I)>\widetilde{C_\delta}$.
This together with the choice of $\{I^{(3)}_j\}$ implies that $m_\lz(I^{(3)}_j)\leq\widetilde{C_\delta}$ for all $j$.
Since for any $j$, $m_\lz(I^{(3)}_j)\sim x_j^{2\lz}r_j>r_j^{2\lz+1}$,
it follows that
$r_j<C\widetilde{C_\delta}^{\frac{1}{2\lz+1}}=:C_\delta.$
Therefore, by the choice of $\{I^{(3)}_j\}_j$, \eqref{I-j-pro} holds. This implies the claim.
%Then there exists a sequence of intervals $\{I^{(3)}_j\}_{j}:=\{I(x^{(3)}_j, r)\}_j$
%for some $r>0$, such that \eqref{lower bdd osci} holds with $x^{(3)}_j\rightarrow\infty$, $j\to\fz$.
%Choose a subsequence $\{x_{j_\ell}^{(3)}\}_{j_\ell}$ of $\{x^{(3)}_j\}_j$ such that the intervals
%$\{I^{(3)}_{j_k}\}_{j_k}$ satisfy that $I^{(3)}_{j_\ell}\cap I^{(3)}_{j_k}=\emptyset$ for $j_k\neq j_\ell$,
%where
%$$I^{(3)}_{j_k}:=\lf\{y\in\rrp: \lf|y-x^{(3)}_{j_k}\r|<A_2r\r\}.$$

Now we define
$$\widetilde{\widetilde{\mathcal J_1}}:=\lf(x_{\ell}+A_1r_{\ell}, x_{\ell}+A_2r_{\ell}\r),$$
%$$\widetilde{\widetilde{\mathcal J_1}}:=\widetilde{\widetilde{\mathcal J}}\setminus \lf\{y\in\mathbb R_+: \lf|y-x_{j_{\ell+m}}^{(3)}\r|\le A_2r\r\},$$
and
$$\widetilde{\widetilde{\mathcal J_2}}:=\lf\{y\in\mathbb R_+: \lf|y-x_{\ell+m}\r|> A_2r_{\ell+m}\r\}.$$
%Note that $(\rrp\setminus\widetilde{\widetilde{\mathcal J_2}})\cap \widetilde{\widetilde{\mathcal J}}=\emptyset$,
Note that $\widetilde{\widetilde{\mathcal J_1}}\subset\widetilde{\widetilde{\mathcal J_2}}$.
Thus, similar to the estimates of ${\rm F_1}$ and ${\rm F_2}$ in Case i), for any $\ell$, $m$, we get
\begin{eqnarray*}
&&\|\lf[b,\riz\r](f_{\ell})-\lf[b,\riz\r](f_{\ell+m})\|_{L^p(\mathbb R_+,\,dm_\lz)}\\
&&\quad\ge\lf\{\int_{\widetilde{\widetilde{\mathcal J_1}}}\lf|\lf[b,\riz\r](f_{\ell})(y)-\lf[b,\riz\r](f_{\ell+m})(y)\r|^p\ytz\r\}^{1/p}\\
&&\quad\ge\lf\{\int_{\widetilde{\widetilde{\mathcal J_1}}}\lf|\lf[b,\riz\r](f_{\ell})(y)\r|^p\ytz\r\}^{1/p}-
\lf\{\int_{\widetilde{\widetilde{\mathcal J_2}}}\lf|\lf[b,\riz\r](f_{\ell+m})(y)\r|^p\ytz\r\}^{1/p}\\
&&\quad\gs\lf(A_3\r)^{1/p}.
\end{eqnarray*}
This contradicts to the compactness of $[b, \riz]$ on $L^p(\mathbb R_+,\,dm_\lz)$,
so $b$ also satisfies condition (iii)
in Theorem \ref{t-cmo char}.
%Similarly, we can prove if $b$ does not
%satisfy (ii) or (iii) of Theorem \ref{t-com char}(please check), then $[b, \riz]$ is not a compact operator.

{\bf Necessity:}

To see the converse, we show that when $b\in\cmoz$, the commutator $[b, \riz]$ is compact on $\lpz$.
For any $\ez>0$, there exists $b_\ez\in\cd$ such that
$$\|b-b_\ez\|_\bmoz<\ez$$
and
$$\|[b, \riz]-[b_\ez, \riz]\|_{\lpz\to\lpz}\ls\|b-b_\ez\|_\bmoz\ls\ez.$$
Thus, it suffices to show that $[b, \riz]$ is a compact operator for $b\in\cd$.

Let $b\in\cd$, to show $[b,\riz]$ is compact on $\lpz$, it suffices to show that for every bounded subset $\mathcal{F}\in\lpz$,
$[b,\riz]\mathcal{F}$ is relatively compact. Thus, we only need to show that $[b,\riz]\mathcal{F}$ satisfies the
conditions (a)---(c) in Theorem \ref{t-fre kol}. We first point out that by Lemma \ref{l-bdd of riz}
and the fact that $b\in \bmoz$, $[b,\riz]$ is bounded on $\lpz$, which implies $[b,\riz]\mathcal{F}$
satisfies (a) in Theorem \ref{t-fre kol}. Next, since $b\in\cd$, by \eqref{cz kernel condition-1}
and the H\"older inequality, there exists $M$ such that
for any $x>M$,
\begin{equation*}
\lf|[b, \riz]f(x)\r|\le |b(x)||\riz f(x)|+|\riz(bf)(x)|\ls \|f\|_\lpz \frac1{m_\lz(I(x, x))}.
\end{equation*}
Hence (b)  in Theorem \ref{t-fre kol} holds for $[b,\riz]\mathcal{F}$.
Therefore, it remains to prove $[b,\riz]\mathcal{F}$ also satisfies (c).

Let $\ez$ be a fixed positive constant in $(0,\frac12)$ and $z\in\rrp$ small enough. Then for any $x\in\rrp$,
\begin{eqnarray*}
&&[b, \riz]f(x)-[b, \riz]f(x+z)\\
&&\quad=\inzf \riz(x, y)[b(x)-b(y)] f(y)\ytz-\inzf \riz(x+z, y)[b(x+z)-b(y)] f(y)\ytz\\
&&\quad=\int_{|x-y|>\ez^{-1} z}\riz(x, y)[b(x)-b(x+z)]f(y)\ytz\\
&&\quad\quad+\int_{|x-y|>\ez^{-1} z}[\riz(x, y)-\riz(x+z,y)][b(x+z)-b(y)]f(y)\ytz\\
&&\quad\quad+\int_{|x-y|\le\ez^{-1} z}\riz(x,y)[b(x)-b(y)]f(y)\ytz\\
&&\quad\quad-\int_{|x-y|\le\ez^{-1} z}\riz(x+z,y)[b(x+z)-b(y)]f(y)\ytz=:\sum_{j=1}^4{\rm L}_i.
\end{eqnarray*}

From \eqref{cz kernel condition-2} and $\ez\in(0, 1/2)$, it follows that
$$|{\rm L}_2|\ls |z|\int_{|x-y|>\ez^{-1}z}\frac{|f(y)|}{m_\lz(I(x, |x-y|))|x-y|}\,\ytz.$$
By this and the H\"older inequality, we have that
\begin{eqnarray*}
\inzf|{\rm L_2}|^p\xtz&\ls&z^p\inzf\lf[\inzf\lf(\frac{\chi_{|x-y|>\ez^{-1}z}(y)}
{|x-y|m_\lz(I(x, |x-y|))}\r)^{1/p'+1/p}|f(y)|\ytz\r]^p\xtz\nonumber\\
&\ls&z^p\inzf\lf\{\lf[\int_{|x-y|>\ez^{-1} z}\frac{y^{2\lz}}{|x-y|m_\lz(I(x, |x-y|))}\,dy\r]^{p/p'}\r.\noz\\
&\quad&\quad\times\lf.\int_{|x-y|>\ez^{-1} z}\frac{|f(y)|^p}{|x-y|m_\lz(I(x, |x-y|))}y^{2\lz}\,dy\r\}\,\xtz\noz\\
&\ls&z^p(\ez z^{-1})^{(p/p')+1}\|f\|^p_\lpz\ls \ez^p\|f\|^p_\lpz,
\end{eqnarray*}
where the last-to-second inequality follows from the fact that
\begin{eqnarray*}
&&\int_{|x-y|>\ez^{-1} z}\frac{y^{2\lz}}{|x-y|m_\lz(I(x, |x-y|))}\,dy\\
&&\quad\sim\sum_{k=0}^\fz\frac1{2^k\ez^{-1}z}\int_{2^k\ez^{-1}z<|x-y|\le2^{k+1}\ez^{-1}z}\frac{y^{2\lz}}{m_\lz(I(x, 2^k\ez^{-1}z))}\,dy\\
&&\quad\ls\sum_{k=0}^\fz\frac1{2^k\ez^{-1}z}\frac{m_\lz(I(x, 2^{k+1}\ez^{-1}z))}{m_\lz(I(x, 2^k\ez^{-1}z))}\ls \ez z^{-1}.
\end{eqnarray*}
By \eqref{cz kernel condition-1}, the fact that $b\in\cd$ and the mean value theorem, we conclude that
$$|{\rm L}_3|\ls \int_{|x-y|\le\ez^{-1} z}\frac{|x-y|}{m_\lz(I(x, |x-y|))}|f(y)|\,\ytz$$
and
$$|{\rm L}_4|\ls \int_{|x-y|\le\ez^{-1} z}\frac{|x+z-y|}{m_\lz(I(x+z, |x+z-y|))}|f(y)|\,\ytz.$$
Then by the fact that
\begin{eqnarray*}
&&\int_{|x-y|\le\ez^{-1} z}\frac{|x-y|}{m_\lz(I(x, |x-y|))}y^{2\lz}\,dy\\
&&\quad\sim\sum_{k=-\fz}^{-1}2^k\ez^{-1}z\int_{2^k\ez^{-1}z<|x-y|\le2^{k+1}\ez^{-1}z}\frac{y^{2\lz}}{m_\lz(I(x, 2^k\ez^{-1}z))}\,dy\\
&&\quad\ls\sum_{k=-\fz}^{-1}2^k\ez^{-1}z\frac{m_\lz(I(x, 2^{k+1}\ez^{-1}z))}{m_\lz(I(x, 2^k\ez^{-1}z))}\ls \ez^{-1} z,
\end{eqnarray*}
we see that
\begin{eqnarray*}
\inzf|{\rm L_3}|^p\xtz&\ls&\inzf\lf\{\lf[\int_{|x-y|\le\ez^{-1} z}\frac{|x-y|}{m_\lz(I(x, |x-y|))}y^{2\lz}\,dy\r]^{p/p'}\r.\noz\\
&\quad&\quad\times\lf.\int_{|x-y|\le\ez^{-1} z}\frac{|x-y||f(y)|^p}{m_\lz(I(x, |x-y|))}y^{2\lz}\,dy\r\}\,\xtz\noz\\
&\ls&(\ez^{-1} z)^{p}\|f\|^p_\lpz,
\end{eqnarray*}
and
\begin{eqnarray*}
\inzf|{\rm L_4}|^p\xtz&\ls&\inzf\lf\{\lf[\int_{|x+z-y|\le{\ez^{-1} z+z}}\frac{|x+z-y|}{m_\lz(I(x+z, |x+z-y|))}y^{2\lz}\,dy\r]^{p/p'}\r.\noz\\
&\quad&\quad\times\lf.\int_{|x+z-y|\le{\ez^{-1} z+z}}\frac{|x+z-y||f(y)|^p}{m_\lz(I(x+z, |x+z-y|))}y^{2\lz}\,dy\r\}\,\xtz\noz\\
&\ls&(\ez^{-1} z+z)^{p}\|f\|^p_\lpz\\
&\ls&(\ez^{-1}z)^{p}\|f\|^p_\lpz.
\end{eqnarray*}

Moreover, observe that
$$|{\rm L}_1|\le |b(x)-b(x+z)|\sup_{t>0}\lf|\int_{|x-y|>t}\riz(x, y)f(y)\ytz\r|=:|b(x)-b(x+z)| \riz_\ast f(x).$$
Since by i) and ii) of Lemma \ref{l-RieszCZ}, $\riz(x,y)$ is a Calder\'on-Zygmund kernel in space of homogeneous type,
we see that $\riz_\ast$ is bounded on $\lpz$ for any $p\in(1, \fz)$; see, for example, \cite{hyy} and \cite{bfbmt}.
Then we have that
\begin{equation*}
\inzf|{\rm L_1}|^p\xtz\ls\inzf\lf[|b(x)-b(x+z)| \riz_\ast f(x)\r]^p\xtz.
\end{equation*}
As $b$ is uniformly continuous, by letting $z$ small enough depending on $\epsilon$, we have that
\begin{equation*}
\inzf|{\rm L_1}|^p\xtz\ls\ez^p\|f\|_\lpz^p.
\end{equation*}
Combining the estimates of ${\rm L}_i,\,i\in\{1, 2,3,4\}$, we conclude that
\begin{eqnarray*}
&&\lf[\inzf\lf|[b, \riz]f(x)-[b, \riz]f(x+z)\r|^p\xtz\r]^{1/p}\\
&&\quad\ls\sum_{i=1}^4\lf(\inzf|{\rm L}_i|^p\xtz\r)^{1/p}\ls\ez\|f\|_\lpz.
\end{eqnarray*}
This shows that $[b,\riz]\mathcal{F}$ satisfies the condition (c) in Theorem \ref{t-fre kol}. Hence,
$[b, \riz]$ is a compact operator. This finishes the proof of Theorem \ref{t-riesz compact}.
\end{proof}

{\bf Acknowledgement:} The authors would like to thank the referee for careful reading and checking, and for all the helpful
suggestions and comments, which helps to make this paper more readable.

%\medskip

%Xuan Thinh Duong

\smallskip

Department of Mathematics, Macquarie University, NSW, 2109, Australia.

\smallskip

{\it E-mail}: \texttt{xuan.duong@mq.edu.au}

\vspace{0.3cm}

%Ji Li

%\smallskip

Department of Mathematics, Macquarie University, NSW, 2109, Australia.

\smallskip

{\it E-mail}: \texttt{ji.li@mq.edu.au}

\vspace{0.3cm}

%Brett D. Wick

%\smallskip

%Department of Mathematics, Washington University--St. Louis, St. Louis, MO 63130-4899 USA
%
%\smallskip
%
%{\it E-mail}: \texttt{wick@math.wustl.edu}
%

School of Mathematical Sciences, Xiamen University, Xiamen 361005,  China
\smallskip

{\it E-mail}: \texttt{suzhen.860606@163.com}

\vspace{0.3cm}

School of Mathematical Sciences, Xiamen University, Xiamen 361005,  China
\smallskip

{\it E-mail}: \texttt{huoxwu@xmu.edu.cn}

\vspace{0.3cm}
%Dongyong Yang

%\smallskip

School of Mathematical Sciences, Xiamen University, Xiamen 361005,  China

\smallskip

{\it E-mail}: \texttt{dyyang@xmu.edu.cn }
\end{document}